\documentclass[a4paper,12pt]{amsart}
\usepackage{geometry}                		% See geometry.pdf to learn the layout options. There are lots.
\usepackage{graphicx}				% Use pdf, png, jpg, or eps§ with pdflatex; use eps in DVI mode
\usepackage{amssymb}

\usepackage{latexsym,exscale,enumerate,amsfonts,amssymb,mathtools}
\usepackage{amsmath,amsthm,amsfonts,amssymb,amscd,stmaryrd,textcomp}
\usepackage[normalem]{ulem}
\usepackage{young}
\usepackage{thmtools,xcolor}
\usepackage{easybmat}

\definecolor{colormy}{rgb}{0.8,0.05,0.05}
\definecolor{mycolor}{rgb}{0.25,0.99,0.25}

\usepackage[all]{xy}
\SelectTips{cm}{}

\usepackage{tikz}
\usetikzlibrary{decorations.markings}
\usetikzlibrary{decorations.pathreplacing}
\usetikzlibrary{arrows,shapes,positioning}
\tikzstyle directed=[postaction={decorate,decoration={markings,
    mark=at position #1 with {\arrow{>}}}}]
\tikzstyle rdirected=[postaction={decorate,decoration={markings,
    mark=at position #1 with {\arrow{<}}}}]

\usepackage{hyperref}

\newcommand{\Hom}{\mathrm{Hom}}
\newcommand{\End}{\mathrm{End}}

\newcommand{\ord}{\mathrm{ord}}
\newcommand{\mo}{\mathrm{mod }}

\def\Z{{\mathbb Z}}

\def\Q{{\mathbb Q}}

\def\F{{\mathcal F}}
\def\Fu{{\mathbb F}_p}

\def\Ne{\mathcal N}

\def\pr{\mathrm{pr}}
\def\Ind{\mathrm{Ind}}

\def\Tr{\mathrm{Tr}}

\theoremstyle{definition}
\newtheorem{thm}{Theorem}[section]
\newtheorem{cor}[thm]{Corollary}
\newtheorem{lem}[thm]{Lemma}
\newtheorem{prop}[thm]{Proposition}

\theoremstyle{definition}
\newtheorem{examplecounter}{Example}
\newtheorem{remarkcounter}{Remark}

\numberwithin{equation}{section}

\title{Higher Jones algebras and their simple modules}

\author{Henning Haahr Andersen}

\email{h.haahr.andersen@gmail.com}

\date{}							% Activate to display a given date or no date

\begin{document}

\begin{abstract}
Let $G$ be a connected reductive algebraic group over a field of positive characteristic $p$ and denote by $\mathcal T$ the category of tilting modules for $G$. The higher Jones algebras are the endomorphism algebras of objects in the fusion quotient category of $\mathcal T$. We determine the simple modules and their dimensions for these semisimple algebras as well as their quantized analogues. This provides a general approach for determining various classes of simple modules for many well-studied algebras such as group algebras for symmetric groups, Brauer algebras, Temperley--Lieb algebras, Hecke algebras and $BMW$-algebras. We treat each of these cases in some detail and give several examples.

\end{abstract}

\maketitle

\section{Introduction}
In \cite{ILZ} the authors define Jones algebras as certain quotients of Temperley--Lieb algebras. They also show that these algebras may be identified with the endomorphism algebras over the quantum group for $\mathfrak sl_2$ of fusion tensor powers of its natural vector representation. In this paper we study more generally such endomorphism algebras for arbitrary reductive groups in positive characteristics and their quantized root of unity analogues.  We call these semisimple algebras {\it higher Jones algebras}.  In the quantum $\mathfrak sl_2$-case they coincide with the Jones algebras from \cite{ILZ}. Our main result is an algorithm which determines the dimensions of the simple modules for these higher Jones algebras. 

As the first important example and as a role model for our study consider the general linear group $GL_n$. Together with the corresponding root of unity quantum case we show how this gives interesting semisimple quotients of the group algebras for symmetric groups as well as of the corresponding Hecke algebras, and it allows us to determine the dimensions of certain classes of simple modules for these algebras. The results in these cases were obtained a long time ago, see \cite{Ma1} and \cite{Kl} for the modular case (using similar, respectively different techniques), and \cite{GW} for the Hecke algebra case. Nevertheless we treat this case in some detail as it will serve as  role model for the general case. In fact, one of our key points is to give a unified treatment of a number of other cases, showing that they can be handled in similar ways.

In order to explain our general strategy we pass now to an arbitrary reductive algebraic group G defined over a field k of prime
characteristic. 
The category of tilting modules for $G$ has a quotient category called the fusion category, see \cite{A92}, \cite{AS}. Objects in this category may be identified with certain semisimple modules for $G$ and the higher Jones algebras are then defined as the endomorphism algebras of such objects and of their fusion tensor powers.  The main result in \cite {AST1} says that any endomorphism algebra of a tilting module for $G$ has a natural cellular algebra structure. We show that the higher Jones algebras inherite a cellular structure, and exploiting this we are able to compute the dimensions of their simple modules. This applies to the corresponding quantum case as well. 

If $G$ is one of the classical groups $GL(V), SP(V)$ or $O(V)$ the module $V$ is (except for $O(V)$ in characteristic $2$ when $\dim V$ is odd) a tilting module for $G$ and via Schur--Weyl duality the endomorphism algebras of the fusion tensor powers of $V$ lead to semisimple quotient algebras of the group algebras of symmetric groups and Brauer algebras. This should  be compared with the results in \cite{Ma1} and \cite{We2}, respectively. In the corresponding quantum cases we obtain semisimple quotients of Hecke algebras and $BMW$-algebras (compare \cite{Kl},  \cite{We1}, and \cite{We3}). In all cases we obtain effective algorithms for computing the dimensions
of the corresponding classes of simple modules. We have illustrated this by giving a number of concrete examples.

We want to point out that our approach, based on the theory of tilting modules and the cellular structure on their endomorphism rings as developed in \cite{AST1}, gives a general method for handling many more cases than those mentioned above. To mention just one big family of algebras - which fits into our framework and which are similar in principle to the examples given so far - take again an arbitrary reductive algebraic group $G$ and let $V$ be a simple Weyl module for $G$. This could e.g. be a Weyl module with minuscule highest weight or with highest weight belonging to the bottom alcove of the dominant chamber. Then our approach applies to the fusion tensor powers of $V$ or more generally to fusion tensor products of finite families of simple Weyl modules. As a result we get algorithms for the dimensions of certain simple modules for the corresponding endomorphism algebras. However, in only few cases are these algebras related to ``known" algebras and we have chosen to limit ourselves to the above examples.

The paper is organized as follows. In Section 2 we first set up notation etc. for a general reductive group $G$ and then make this explicit in the case where $G$ is a group a classical type. We shall rely heavily on tilting modules for $G$ and in Section 3 we start out by recalling the basic facts that we will need from this theory. In addition, this section establishes results on fusion tensor products which we then apply in Section 4 to symmetric groups and in Section 5 to Brauer algebras. In Section 6 we turn to quantum groups at roots of unity. Here we prove results analogous to the ones we obtained in the modular case  and in Sections 7 and 8 we apply these to Hecke algebras for symmetric groups and to $BMW$-algebras, respectively.

\vskip .5 cm 
{\bf Acknowledgments. }
Thanks to the referee for a quick and careful reading as well as for her/his many useful comments and corrections.

\section{Reductive algebraic groups} 
This section introduces notation and contains a few basic facts about reductive algebraic groups and their representations over a field of prime characteristic. We shall be rather brief and refer the reader to \cite{RAG} for further details. We also deduce some specific facts needed later on for each of the groups of classical type.

\subsection {General notation} \label{general}
Suppose $G$ is a connected reductive algebraic group over a field $k$. We assume $k$ has characteristic $p > 0$. In this paper all modules will be finite-dimensional.
 
Let $T$ be a maximal torus in $G$, and denote by $X =X(T)$ its character group. In the root system $R \subset X$ for $(G,T)$ we choose a set of positive roots $R^+$, and denote by $X^+ \subset X$  the corresponding cone of dominant characters. Then $R^+$ defines an ordering $\leq$ on $X$. It also determines uniquely a Borel subgroup $B$ whose roots are the set of negative roots $-R^+$.

Denote by $S$ the set of simple roots in $R^+$. The reflection $s_\alpha$ corresponding to $\alpha \in S$ is called a simple reflection. The set of simple reflections generates the Weyl group $W$  for $R$. We can identify $W$ with 
$N_G(T)/T$. Then we see that $W$ acts naturally on $X$:  $\lambda \mapsto w(\lambda), \lambda \in X, w \in W$. 
%We denote by $\sgn: W \rightarrow \{\pm 1\}$ the sign function on $W$. If $w \in W$ then $\sgn(w)$ is $1$, respectively $-1$ if $w$ is a product of an even, respectively odd,  number of simple reflections. 
In addition to this action of $W$ on $X$ we shall also consider the so-called dot-action given by: $w \cdot \lambda = w(\lambda + \rho) - \rho, w \in W, \lambda \in X$. As usual, $\rho$ is half the sum of the positive roots.

In the category of $G$-modules we have the family of standard modules $\Delta(\lambda)$, and likewise the family of costandard modules $\nabla(\lambda)$. Here $\lambda$ runs through the set of dominant weights $X^+$ and $\Delta(\lambda)$ is also known as  the Weyl module with highest weight $\lambda$. The dual Weyl module $\nabla(\lambda)$ is then $\Delta(-w_0 \lambda)^*$ where $w_0$ denotes the longest element in $W$. 

The simple module with highest weight $\lambda$ may be realized as the head of $\Delta(\lambda)$ as well as the socle of $\nabla(\lambda)$. Recall that there is up to scalars a unique non-zero homomorphism 
\begin{equation} \label{can} c_\lambda: \Delta(\lambda) \rightarrow \nabla(\lambda),
\end{equation}
namely the one factoring through $L(\lambda)$.

A $G$-module $M$ is said to have a $\Delta$-filtration if it has submodules $M^{i}$ with 
$$ 0=M^0 \subset M^1 \subset \dots \subset M^r = M, \text { where } M^{i+1}/M^{i} \simeq \Delta(\lambda_i) \text { for some  } \lambda_i \in X^+.$$
One defines $\nabla$-filtrations similarly.

If $M$ has a $\Delta$-filtration we set $(M:\Delta(\mu))$ equal to the number of occurrences of $\Delta(\mu)$ in such a filtration (note that these numbers are uniquely determined and independent of which $\Delta$-filtration we choose). When $M$ has a $\nabla$-filtration the numbers $(M:\nabla(\mu))$ are defined analogously.

A crucial result concerning modules with a $\Delta$-filtration says that, if $M$ and $M'$ both have a $\Delta$-filtration, then so does $M\otimes M'$. This is the Wang--Donkin-
-Mathieu theorem, see \cite{Wa}, \cite{Do-book}, and \cite{Ma}.

For $n \in \Z$ and $\alpha \in S$ we denote by $s_{\alpha, n}$ the affine reflection determined by
$$s_{\alpha, n}(\lambda) = s_\alpha (\lambda) - np\alpha.$$

The affine Weyl group $W_p$ is the group generated by all $s_{\alpha, n}$ where $ \alpha \in S$ and $ n \in \Z$ (note that in the Bourbaki convention this is the affine Weyl group corresponding to the dual root systen $R^\vee$).

The linkage principle \cite{A80a} says that, whenever $L(\lambda)$ and $L(\mu)$ are two composition factors of an indecomposable $G$- module, then $\mu \in W_p \cdot \lambda$. It follows that $M$ splits into a direct sum of submodules according to the orbits of $W_p$ in $X$. More precisely, if we set
 $$A(p)= \{\lambda \in X  | 0 < \langle \lambda + \rho, \alpha^\vee \rangle < p \text { for all } \alpha \in R^+\},$$ 
called the bottom dominant alcove, then the closure 
$$\overline A(p) =  \{\lambda \in X  | 0 \leq \langle \lambda + \rho, \alpha^\vee \rangle \leq p \text { for all } \alpha \in R^+\}$$
 is a fundamental domain for the dot-action of $W_p$ on $X$. We have
$$ M = \bigoplus_{\lambda \in \overline A(p)} M[\lambda], $$
with $M[\lambda]$ equal to the largest submodule in $M$ whose composition factors $L(\mu)$  all have $\mu \in W_p \cdot \lambda$. 

\begin{remarkcounter} \label{alcove A}
\begin{enumerate}
\item [a)]
As an immediate consequence of the strong linkage principle \cite{A80a} we have
$$ \Delta(\lambda) = L(\lambda) = \nabla(\lambda)  \text { for all } \lambda \in \overline A \cap X^+.$$
\item [b)] We have $A(p) \neq \emptyset$ if and only if $p >  \langle \rho, \alpha^\vee \rangle$ for all roots $\alpha$, i.e. if and only if $p \geq h$, where $h$ is the Coxeter number for $R$.
\end{enumerate}
\end{remarkcounter}

\subsection{The general linear groups} \label{GL}
Let $V$ be a vector space over $k$. The reductive group $GL(V)$ plays a particularly important role in this paper. In this section we make the above notations and remarks explicit for the group $GL(V)$.

We set $n = \dim V$ and choose a basis $\{v_1, v_2, \cdots , v_n\}$ for $V$. Then $G_n = GL(V)$ identifies with $GL_n(k)$ and the set $T_n$ of diagonal matrices in $G_n$ is a maximal torus. The character group $X_n = X(T_n)$ is the free abelian group with basis $\epsilon_i$, $i=1, 2, \cdots ,n$ where $\epsilon_i: T_n \rightarrow k^{\times}$ is the homomorphism mapping a diagonal matrix into its $i$'th entry. If $\lambda \in X_n$, we shall write
$$\lambda = (\lambda_1, \lambda_2, \cdots , \lambda_n),$$ 
when $\lambda = \sum_1^n \lambda_i \epsilon_i.$ 
The root system for $(G_n,T_n)$ is 
$$R = \{\epsilon_i -\epsilon_j | i \neq j\}.$$ 
It is of type $A_{n-1}$. Our choice of simple roots $S$ will be
$$S = \{\alpha_i = \epsilon_i -\epsilon_{i+1} | i = 1, 2, \cdots , n-1\}$$
inside the set of positive roots $R^+$ consisting of all $\epsilon_i - \epsilon_j$ with $i<j$. 

We set 
$$\omega_i = \epsilon_1 + \epsilon_2 + \cdots + \epsilon_i, \; i = 1, \cdots , n.$$ 
Then $\{\omega_1, \cdots , \omega_n\}$ is another basis of $X_n$. Note that $\omega_n$ is the determinant and thus, is trivial on the intersection of $T_n$ with $SL_n(k)$. Consider 
$$\rho' = \omega_1 + \cdots + \omega_n = (n, n-1, \cdots , 1).$$ 
Then $\rho' = \rho +\frac{1}{2}  (n+1) \omega_n$ and we shall prefer to work with $\rho'$ instead of $\rho$ (note that, if $n$ is even, $\rho \notin X_n$ whereas $\rho' \in X_n$ for all $n$). As $\omega_n$ is fixed by $W$, the dot-action of $W$ on $X$ is unchanged when we replace $\rho$ by $\rho'$.

We have an inner product on $X_n$ given by $(\epsilon_i, \epsilon_j) = \delta_{i,j}$. It satisfies $(\omega_i, \alpha_j) = \delta_{i,j}$, $i,j = 1, 2, \cdots n-1$, i.e. $\omega_1, \cdots , \omega_{n-1}$ are the fundamentals weights in $X_n$. On the other hand, $(\omega_n, \alpha_j) = 0$ for all $j$. Hence, $(\rho', \alpha_j) = 1$ for all $j = 1, 2, \cdots , n-1$.

The set of dominant weights  is 
$$X_n^+ = \{\lambda \in X_n | \lambda_1 \geq \lambda_2 \geq  \cdots \geq \lambda_n\} = \{\sum_1^n m_i \omega_i | m_i \in \Z_{\geq 0}, i= 1, 2, \cdots , n-1, \;\ m_n \in \Z\}.$$ 
If $\lambda \in X_n^+$ has $\lambda_n \geq 0$,
then $\lambda$ may be identified with a partition of $|\lambda| = \lambda_1 + \lambda_2 + \cdots + \lambda_n$.

The bottom alcove will be denoted $A_n(p)$.  When $n > 1$ it is given by
$$ A_n(p) = \{\lambda \in X_n^+ | \lambda_1 - \lambda_n \leq p - n\}.$$
We have $A_n(p) \neq \emptyset$ if and only if $p \geq n$. In particular, $A_2(p)$ is always non-empty.
In the special case $n = 1$ the group $G_1$ is the $1$-dimensional torus. In that case $X_1 = \Z \epsilon_1$ and $A_1(p) = \Z \epsilon_1$. Note that for any $r \in \Z$ the Weyl module $\Delta_1(r \epsilon_1)$ is the $1$-dimensional $G_1$ module  $k_{r\epsilon_1}$.

\begin{remarkcounter} \label{A for type A}
 The natural module $V$ for $G_n$ has weights $\epsilon_1, \cdots , \epsilon_n$. It is simple because the highest weight $\epsilon_1$ is minuscule. We have $\epsilon_1 \in A_n(p)$ if and only if $p > n$. 
\end{remarkcounter}

\subsection{The symplectic groups} \label{Sp}

Let now $V$ be a $2n$-dimensional symplectic vector space over $k$ with a fixed symplectic form, and consider the semisimple algebraic group $G_n = SP(V)$ consisting of those elements in $GL(V)$ which respect this form. This is naturally a subgroup of $GL(V)$. Note that $G_1 = SL_2(k)$.  We let $T_n$ be the maximal torus in $G_n$ obtained as the intersection of the maximal torus in $GL(V)$ with $G_n$. In the notation from Section \ref{GL} the restrictions to $T_n$ of $\epsilon_1, \cdots , \epsilon_n$ form a basis of $X_n = X(T_n)$. The root system for $(G_n, T_n)$ consists of the the elements
$$\{\pm \epsilon_i \pm \epsilon_j, \pm2 \epsilon_i | 1 \leq i \neq j \leq n \},$$ 
and is of type $C_n$. With respect to the usual choice of positive roots the set of dominant weights is 
$$X_n^+ = \{\lambda = \sum_i \lambda_1 \epsilon_i | \lambda_i \geq \lambda_2 \geq \cdots \geq \lambda_n \geq 0 \}.$$

The bottom dominant alcove is also in this case denoted $A_n(p)$. It is given by 
$$ A_n(p) = \begin{cases} \{\lambda \epsilon_1 | 0 \leq \lambda \leq p-2\} \text { if } n = 1, \\\{\lambda \in X_n^+ | \lambda_1 + \lambda_2 \leq p - 2n \} \text { if } n > 1. \end{cases}$$
When $n > 1$ we have $A_n(p) \neq \emptyset$ if and only if $p \geq 2n$, whereas $A_1(p) \neq \emptyset$ for all $p$.

\begin{remarkcounter} \label{A for type C}
The natural module $V$ for $G_n$ is simple for all $p$ as its highest weight $\epsilon_1$ is minuscule. It has weights $\pm \epsilon_1, \cdots , \pm \epsilon_n$. Note that, for $n > 1$, we have $\epsilon_1 \in A_n(p)$ if and only if $p > 2n$, whereas for $n=1$ the condition is $p > 2$.
\end{remarkcounter}

\subsection{The orthogonal groups} \label{O}
Consider next a vector space $V$ over $k$ equipped with a non-degenerate, symmetric bilinear form.  Then the orthogonal group $O(V)$ is the subgroup of $GL(V)$ consisting of those elements which preserve the bilinear form on $V$. We shall separate our discussion into the case where $\dim V$ is odd and the case where $\dim V$ is even. 
\subsubsection{Type $B_n$}
Assume that $\dim V$ is odd, say $\dim V = 2n + 1$. Then we set $G_n = O(V)$. Again in this case we have $G_1 \simeq SL_2(k)$. However, the module $V$ for $G_1$ is the $3$-dimensional standard module for $SL_2(k)$. The root system $R$ for $G_n$ has type $B_n$ and may be taken to consist of the elements 
$$ R = \{\pm \epsilon_i \pm \epsilon_j, \pm \epsilon_i | 1 \leq i \neq j \leq n \}$$
in $X_n = \oplus _{i=1}^n \Z\epsilon_i$. The set of dominant weights is 
$$X_n^+ = \{ \sum_i \lambda_i \epsilon_i \in X_n | \lambda_1 \geq \lambda_2 \geq \cdots \geq \lambda_n \geq 0 \}.$$
In this case the bottom dominant alcove $A_n(p)$ is given by
$$ A_n(p) =  \begin{cases} \{\lambda \epsilon_1 | 0 \leq \lambda \leq p-2\} \text { if } n = 1,  \\\{\lambda \in X_n^+ | 2 \lambda_1 \leq p - 2n \} \text { if } n > 1. \end{cases}$$
We have $A_n(p) \neq \emptyset$ if and only if $p > 2n$ (except for $n = 1$). 
\begin{remarkcounter} \label{A for type B}
Unlike in the previous cases the highest weight of $V$ is no longer minuscule. However,  we still have $V = \Delta(\epsilon_1)$ is simple for all $p > 2$, \cite[Section II.8.21]{RAG}. It has weights $\pm \epsilon_1, \cdots , \pm \epsilon_n$ together with $0$. Note that $\epsilon_1 \in A_n(p)$ if and only if $p > 2n + 2$ except for $n=1$ where the condition is $p > 2$.
\end{remarkcounter}

\subsubsection{Type $D_n$}
Assume that $\dim V$ is even, say $\dim V = 2n$. We set again $G_n = O(V)$. The corresponding root system $R$ then has type $D_n$ and may be taken to consist of the elements  
$$ R = \{\pm \epsilon_i \pm \epsilon_j | 1 \leq i \neq j \leq n \}$$ 
in $X_n =\{\lambda \in \oplus _{i=1}^n \frac{1}{2} \Z\epsilon_i | \lambda_i - \lambda_j \in \Z  \text { for all } i, j\}$. The set of dominant weights is 
$$X_n^+ = \{ \sum_i \lambda_1 \epsilon_i \in X_n | \lambda_i \geq \lambda_2 \geq \cdots \geq \lambda_{n-1} \geq  |\lambda_n| \}.$$
In this case the bottom dominant alcove $A_n(p)$ is given by
$$ A_n(p) =  \begin{cases} \{\lambda \epsilon_1 | \lambda \in Z, 0 \leq \lambda \leq p-2\} \text { if } n = 1,  \\\{\lambda \in X_2^+ |  \lambda_1 \pm \lambda_2 \leq p - 2 \} \text { if } n = 2, \\\{\lambda \in X_n^+ |  \lambda_1 + \lambda_2 \leq p - 2n + 2 \} \text { if } n > 2. \end{cases}$$
When $n>2$ we have $A_n(p) \neq \emptyset$ if only if $p > 2n - 2$, whereas $A_1(p)$ and $A_2(p)$ are always non-empty.

\begin{remarkcounter} \label{A for type D}
$V = \Delta(\epsilon_1)$ is simple for all $p $ (its highest weight is minuscule). It has weights $\pm \epsilon_1, \cdots , \pm \epsilon_n$. Note that, for $n > 2$, we have  $\epsilon_1 \in A_n(p)$ if and only $p > 2n - 2$, whereas the condition for both $n=1$ and $n=2$ is $p > 2$.
\end{remarkcounter}

\section{Tilting modules for reductive algebraic groups}
We return to the situation of a general reductive group $G$ and use the notation from Section \ref{general}.  We very briefly recall the basics about tilting modules for $G$  (referring to \cite[Section 2]{Do} or \cite[ChapterII.E]{RAG} for details), and prove the results which we then apply in the next two sections. Moreover, we recall from \cite[Section 4]{AST1} a few facts about the cellular algebra structure on endomorphism rings for tilting modules for $G$, which we also need. 

\subsection {Tilting theory for $G$}
A $G$-module $M$ is called tilting if it has both a $\Delta$- and a $\nabla$-filtration. It turns out that for each $\lambda \in X^+$ there is a unique (up to isomorphisms) indecomposable tilting module $T(\lambda)$ with highest weight $\lambda$,  and up to isomorphisms these are the only indecomposable tilting modules, see \cite[Theorem 1.1] {Do}.  The Weyl module $\Delta(\lambda)$ is a submodule of $T(\lambda)$, while the dual Weyl module $\nabla(\lambda)$ is a quotient. The composite of the inclusion $\Delta(\lambda) \to T(\lambda)$ and the quotient map $T(\lambda) \to \nabla(\lambda)$ equals the homomorphism  $c_\lambda$ from (\refeq{can}) (up to a non-zero constant in $k$).

We have  the following elementary (and no doubt well-known) lemma.

\begin{lem} \label{quotient} Let $M$ be a $G$-module which contains two submodules $M_1$ and $M_2$ such that $M = M_1 \oplus M_2$. Denote by $i_j: M_j \rightarrow M$, respectively $\pi_j: M \rightarrow M_j$ the natural inclusion, respectively projection, $j = 1,2$. Suppose $f \circ g = 0$ for all $f \in \Hom_G(M_2, M_1)$ and $g \in \Hom_G(M_1, M_2)$. Then the natural map 
$$\phi: \End_G(M) \rightarrow \End_G(M_1)$$
which takes  $h \in \End_G(M)$ into $\pi_1 \circ h \circ i_1 \in \End_G(M_1)$
is a surjective algebra homomorphism.
\end{lem}

\begin{proof} The surjectivity of $\phi$ is obvious, so we just have to check that $\pi_1 \circ h' \circ h \circ i_1 = \pi_1 \circ h' \circ i_1 \circ \pi_1 \circ h \circ i_1$ for all $h', h \in \End_G(M)$. However, $h \circ i_1 = i_1 \circ \pi_1\circ h \circ i_1 +  i_2 \circ \pi_2\circ h \circ i_1$, and by our assumption we see that $(\pi_1 \circ h' \circ i_2) \circ  (\pi_2\circ h \circ i_1) = 0$. The desired equality follows.
\end{proof}

Let $Q$ be a tilting module for $G$. Then $Q$ splits into indecomposable summands as follows
$$ Q = \bigoplus_{\lambda \in X^+} T(\lambda)^{(Q:T(\lambda))}$$
for unique $(Q:T(\lambda)) \in \Z_{\geq 0}$. 

Set now 
$$Q^{\F}= \bigoplus_{\lambda  \in A(p)} T(\lambda)^{(Q:T(\lambda))} \text { and } Q^{\Ne} = \bigoplus_{\lambda  \in X^+\setminus A(p)}T(\lambda)^{(Q:T(\lambda))}.$$ 
For reasons explained in Section \ref{Fusion} we call $Q^{\F}$ the fusion summand of $Q$ and $Q^{\Ne}$ the negligible summand of $Q$. Then:

\begin{lem} \label{no composites} 
If $f \in \Hom_G(Q^{\F}, Q^{\Ne})$ and $g \in \Hom_G(Q^{\Ne}, Q^{\F})$, then $g \circ f = 0$.

\end{lem}

\begin{proof} It is enough to check the lemma in the case where $Q^{\F} = T(\lambda)$ and $Q^{\Ne} = T(\mu)$ with $\lambda \in A(p)$ and $\mu \in X^+\setminus A(p)$. By Remark \ref{alcove A}a) we have $T(\lambda) = \Delta(\lambda) = L(\lambda)$. Hence, in this case $g \circ f$ is up to scalar the identity on $T(\lambda)$. If the scalar were non-zero, $T(\lambda)$ would be a summand of $T(\mu)$, which contradicts the indecomposability of $T(\mu)$.
\end{proof}

\begin{thm} \label{fusion-quotient}
The natural map $\phi: \End_G(Q) \rightarrow \End_G(Q^{\F})$ is a surjective algebra homomorphism. The kernel of $\phi$ equals 
$$\{h \in \End_G(Q) | \Tr(i_\lambda \circ h \circ \pi_\lambda) = 0 \text { for all homomorphisms } i_\lambda: T(\lambda) \rightarrow Q, \; \pi_\lambda: Q \rightarrow T(\lambda),\; \lambda \in X^+\}.$$
\end{thm}

\begin{proof} The combination of Lemma \ref{quotient} and Lemma \ref{no composites} immediately gives the first statement. To prove the claim about the kernel of $\phi$ we first observe that the endomorphisms $i_\lambda \circ h \circ \pi_\lambda$ of $T(\lambda)$ are either nilpotent, or a constant times the identity. Now, nilpotent endomorphisms clearly have trace zero. If $ \lambda \notin A(p)$, then $\dim T(\lambda)$ is divisible by $p$. This holds when $\lambda$ is $p$-singular (i.e. if there exists a root $\beta$ with $\langle \lambda + \rho, \beta^{\vee} \rangle$ divisible by $p$), because then the linkage principle implies that $\mu$ is also $p$-singular for all $\mu \in X^+$ for which $\Delta(\mu)$ occurs in a $\Delta$-filtration of $T(\lambda)$. For a $p$-regular $\lambda$ it then follows by an easy translation argument, see \cite[Section 5]{A92} (the argument used there deals with the quantum case but applies just as well in the modular case). So when $\lambda \notin A(p)$ all endomorphisms of $T(\lambda)$ have trace zero. In particular, we have $\Tr(i_\lambda \circ h \circ \pi_\lambda) = 0$ for all $h \in \End_G(Q)$ when $\lambda \in X^+\setminus A(p)$.  

On the other hand, if $\lambda \in A(p)$, then by Remark \ref{alcove A} we see that $T(\lambda)$ is simple, i.e. we have $T(\lambda) = L(\lambda) =  \Delta(\lambda)$. So in this case any non-zero endomorphism of $T(\lambda)$ has trace equal to a non-zero constant times $\dim T(\lambda)$. By Weyl's dimension formula $\dim (\Delta(\lambda)) $ is prime to $p$.  If $i_1$, respectively $\pi_1$, denotes the natural inclusion of $Q^\F$ into $Q$, respectively projection onto $Q^\F$, then this means that $h \in \End_G(Q)$ is in the kernel of $\phi$ if and only if $i_1 \circ h \circ  \pi_1 = 0$ if and only if $i_\lambda \circ h \circ \pi_\lambda = 0$ for all $\lambda \in A(p)$ if and only if   $ \Tr(i_\lambda \circ h \circ \pi_\lambda) = 0$ for all $\lambda \in X^+$.
\end{proof}

\subsection{Fusion} \label{Fusion}

Let $\mathcal T$ denote the category of tilting modules for $G$. As noted above, this is a tensor category. Inside $\mathcal T$ we consider the subcategory $\mathcal N$ consisting of all negligible modules, i.e. a tilting module $M$ belongs to $\mathcal N$ if and only if  $\Tr(f) = 0$ for all $f \in \End_{G}(M)$. As each object in $\mathcal T$ is a direct sum of certain of the $T(\lambda)$'s and $\dim T(\lambda)$ is divisible by $p$  if and only if  $\lambda  \notin A(p)$ (as we saw in the proof of Theorem \ref{fusion-quotient}) we see that $M \in \mathcal N$ if and only if  $(M:T(\lambda)) = 0$ for all $\lambda \in A(p)$.

We proved in \cite[Section 4]{A92} (in the quantum case - the arguments for $G$ are analogous) that $\mathcal N$ is a tensor ideal in $\mathcal T$.  The corresponding quotient category $\mathcal T/ \mathcal N$ is then itself a tensor category. It is denoted $\F$ and called the fusion category for $G$. We may think of objects in $\mathcal F$ as the tilting modules $Q$ whose indecomposable summands are among the $T(\lambda)$'s with $\lambda \in A(p)$. Note that $\mathcal F$ is a semisimple category (with simple objects  $(T(\lambda) = L(\lambda))_{\lambda \in A(p)}$), cf. \cite[Section 4]{A92}.

\begin{remarkcounter} \label{p=2$} Note that for $p < h$ the alcove $A(p)$ is empty.  This means that in this case $\mathcal N = \mathcal T$. In particular, if $p=2$ the fusion category is trivial except for the case $G=SL_2(k)$ in which case $A(2) =\{0\}$, so that $\mathcal F$ is the category of finite-dimensional vector spaces. For this reason we shall in the following tacitly assume $p > 2$.
\end{remarkcounter}

In order to distinguish it from the usual tensor product on $G$-modules we denote the tensor product in $\mathcal F$ by $\underline \otimes$. If $Q, Q' \in \mathcal F$ then $Q \underline \otimes Q = \pr (Q \otimes Q')$ where $\pr$ denotes the projection functor from $\mathcal T$ to $\mathcal F$ (on the right-hand side we consider $Q, Q'$ as modules in $\mathcal T$).

\begin{cor} \label{fusion} Let $T$ be an arbitrary tilting module for $G$. Then, for any $r \in \Z_{\geq 0}$,  the natural homomorphism $\End_G(T^{\otimes r}) \rightarrow \End_G(T^{\underline \otimes r})$ is surjective.
\end{cor}

\begin{proof} Set $Q = T^{\otimes r}$. Then $Q$ is a tilting module and in the above notation $Q^\F = T^{\underline \otimes r}$ ($ = 0$ if $T \in \mathcal N$) . Hence the corollary is an immediate consequence of Theorem \ref{fusion-quotient}.
\end{proof}

\subsection{Cellular theory for endo-rings of tilting modules} \label{cellular}

Recall the notions of cellularity, cellular structure and cellular algebras from \cite{GL}. When $Q$ is a tilting module for $G$ its endomorphism ring $E_Q = \End_G(Q)$ has a natural cellular structure, see \cite[Theorem 3.9] {AST1}. The parameter set of weights for $E_Q$ is 
$$ \Lambda = \{\lambda \in X^+ | (Q:\Delta(\lambda)) \neq 0 \}.$$

When $\lambda \in X^+$ the cell module for $E_Q$ associated with $\lambda$ is $C_Q(\lambda) = \Hom_G(\Delta(\lambda), Q)$. Then $\dim C_Q(\lambda) = (Q:\Delta(\lambda))  (= 0$ unless $\lambda \in \Lambda$). We set 
$$ \Lambda_0 = \{\lambda \in \Lambda | (Q:T(\lambda)) \neq 0 \}.$$
If $\lambda \in \Lambda_0$ then $C_Q(\lambda)$ has a unique simple quotient which we in this paper denote $D_Q(\lambda)$. The set $\{D_Q(\lambda) | \lambda \in \Lambda_0\}$ is up to isomorphisms the set of simple modules for $E_Q$. We have 
\begin{equation} \label{dim simple/tilting}
\dim D_Q(\lambda) = (Q: T(\lambda)), \end{equation}
see \cite[Theorem 4.12]{AST1}.

Finally, recall the following result on semisimplicity,  see  \cite[Theorem 4.13] {AST1}.
\begin{thm} \label{ss} $E_Q$ is a semisimple algebra if and only if  $Q$ is a semisimple $G$-module. In that case we have $\Lambda = \Lambda_0$,  $T(\lambda) = \Delta(\lambda) = L(\lambda)$ and $C_Q(\lambda) = D_Q(\lambda)$ for all $\lambda \in \Lambda$.
\end{thm}

\begin{examplecounter} \label{Ex1} Let $T$ be a tilting module for $G$ and set $Q = T^{\otimes r}$ as in the previous section. Then $E_Q= \End_G(T^{\otimes r})$ is a cellular algebra with cell modules $(C_Q(\lambda))_{\lambda \in \Lambda_T^r}$, where $\Lambda_T^r = \{\lambda \in X^+ | (T^{\otimes r} : \Delta(\lambda)) \neq 0\}$, and simple modules $(D_Q(\lambda))_{\lambda \in \Lambda_{0,T}^r}$, $\Lambda_{0,T}^r = \{\lambda \in \Lambda_T^r | (T^{\otimes r}: T(\lambda)) \neq 0\}$. 

Denote by $Q_1$ the summand  $T^{\underline \otimes r}$ of $Q$.   Then the endomorphism ring 
$$\overline E_Q = \End_G(Q_1),$$
is,  according to Corollary \ref{fusion}, a quotient of $E_Q$, and, by Theorem \ref{ss}, it is a semisimple cellular algebra. In fact, $\overline E_Q$ is a direct sum of the matrix rings, namely
$$\overline E_Q \simeq \bigoplus_{\lambda \in A(p)}  M_{(Q_1:T(\lambda))} (k).$$
The simple modules for $\overline E_Q$ are $\{D_Q(\lambda) | \lambda \in A(p) \cap \Lambda_0 \}$. We have 
$$\dim D_Q(\lambda) = (Q_1:T(\lambda))$$
for all $\lambda \in A(p) \cap \Lambda_{0,T}^r$.

\end{examplecounter}

\section {Semisimple quotients of the group algebras $kS_r$}
 
In this section $G_n = GL(V)$ where $V$ is a vector space over $k$ of dimension $n$ with basis $\{v_1, v_2, \cdots , v_n\}$ as in Section \ref{GL}. As we will also look at various  subspaces $V'$ of $V$ we shall from now on write $V_n = V$. We write  $\Delta_n(\lambda)$ for the Weyl module for $G_n$, $T_n(\lambda)$  for the indecomposable tilting module for $G_n$ with highest weight $\lambda$, etc.

\subsection{Algebras associated with tensor powers of the natural module for $G$} \label{tensor powers}

We let $r \in Z_{\geq 0}$ and consider the $G_n$-module $V_n^{\otimes r}$. As tensor products of tilting modules are again tilting we see that this is a tilting module for $G_n$. Consider the subset $I = \{\alpha_1, \alpha_2, \cdots , \alpha_{n-2}\}  \subset S$. Then the corresponding Levi subgroup $L_I$ identifies with $G_{n-1} \times G_1$, where the first factor $G_{n-1} =GL(V_{n-1})$ is the subgroup fixing $v_n$, and  the second factor $G_1$ is the subgroup fixing $v_i$ for $i < n$ and stabilizing the line $k v_n$. As an $L_I$-module we have $V_n = V_ {n-1} \oplus k_{ \epsilon_n}$. Here $k_{\epsilon_n}$ is the line $k v_n$ on which $G_1$ acts via the character $\epsilon_n$. This gives 
\begin{equation} \label{restriction}
V_n^{\otimes r} \simeq \bigoplus _{s=0} ^r (V_{n-1}^{\otimes s} \otimes k_{(r-s)\epsilon_n})^{\oplus \binom{r}{s}} \text { (as $L_I$-modules)}. 
\end{equation}
In particular, $V_{n-1}^{\otimes r}$  is an $L_I$-summand. Its weights (for the natural maximal torus in $L_I$ which is also the maximal torus $T_n$ in $G_n$) consist of $\lambda$'s with $\lambda_n = 0$ whereas any weight $\mu$ of the complement $C =  \bigoplus _{s=0} ^{r-1} (V_{n-1}^{\otimes s} \otimes k_{(r-s)\epsilon_n})^{\oplus \binom{r}{s}}$ has $\mu_n > 0$. It follows that
\begin{equation} \label{no cross-homs} 
 \Hom_{L_I}(V_{n-1}^{\otimes r}, C) = 0 = \Hom_{L_I}(C, V_{n-1}^{\otimes r}).
\end{equation}
Moreover, since $G_1$ acts trivially on $V_{n-1}$ we have $\End_{L_I}(V_{n-1}^{\otimes r}) = \End_{G_{n-1}}(V_{n-1}^{\otimes r})$. Hence we get from Lemma \ref{quotient} (in which the assumptions are satisfied because of (\refeq{no cross-homs})):
\begin{prop}\label{surj GL} The natural algebra homomorphism
$$\End_{G_n}(V_n^{\otimes r}) \rightarrow \End_{G_{n-1}}(V_{n-1}^{\otimes r})$$ 
is surjective.
\end{prop}

Later on we shall use the following related result.
\begin{prop} \label{restriction of Specht}
Suppose $\lambda \in X^+$ has $\lambda_n = 0$. Then the natural homomorphism 
$$\Hom_{G_n}(\Delta_n(\lambda), V_n^{\otimes r}) \rightarrow \Hom_{G_{n-1}}(\Delta_{n-1}(\lambda), V_{n-1}^{\otimes r})$$ 
is an isomorphism for all $r$.
\end{prop}

\begin{proof} In this proof we shall need the parabolic subgroup $P_I$ corresponding to $I$. We have $P_I = L_I U^{I}$ (semidirect product) where $U^{I}$ is the unipotent radical of $P_I$. We set $\nabla_I(\lambda) = \Ind_B^{P_I}(k_\lambda)$. Our assumption that $\lambda_n = 0$ implies that as an $L_I$-module and as a $G_{n-1}$-module we have $\nabla_I(\lambda) = \nabla_{n-1}(\lambda)$.

We shall prove the proposition by proving the dual statement 
$$\Hom_{G_n}( V_n^{\otimes r}, \nabla_n(\lambda)) \simeq \Hom_{G_{n-1}}(V_{n-1}^{\otimes r}, \nabla_{n-1}(\lambda)).$$ 
First, by Frobenius reciprocity \cite[Proposition I.3.4]{RAG}, we have
$$\Hom_{G_n}( V_n^{\otimes r}, \nabla_n(\lambda)) \simeq \Hom_{P_I}( V_n^{\otimes r}, \nabla_I(\lambda)).$$ 
Then restricting to $L_I$ gives an isomorphism to $\Hom_{L_I}( V_n^{\otimes r}, \nabla_I(\lambda))$. Finally, we use (\refeq{restriction}) and the weight arguments from the proof of Proposition \ref{surj GL} to see that this identifies with 
$$\Hom_{L_I}( V_{n-1}^{\otimes r}, \nabla_I (\lambda)) \simeq \Hom_{G_{n-1}}( V_{n-1}^{\otimes r}, \nabla_{n-1}(\lambda)).$$
\end{proof}

We can, of course, iterate the statement in Proposition \ref{surj GL}: If we  set $E_n^r = \End_{G_n}(V_n^{\otimes r})$, then we recover the following well-known fact (cf.  \cite[E.17]{RAG}).
\begin{cor} \label{sequence of surjections} We have a sequence of surjective algebra homomorphisms
$$ E_n^r \rightarrow E_{n-1}^r \rightarrow \cdots \rightarrow E_2^r \rightarrow E_1^r.$$
\end{cor}
%Note that $E_1^r = \End_{G_1}(k_{r \epsilon_1}) = k$ for all $r$. When $n \geq r$ Schur--Weyl duality says that $E_n^r \simeq kS_r$.

\vskip 1 cm

Set now $\overline E_n^r = \End_{G_n}(V_n^{\underline \otimes r})$. Note that these are the higher Jones algebras (see the introduction) in the case $G= GL_n$  corresponding to the tilting modules $V_n^{\otimes r}$. We get from Corollary \ref{fusion} that this is a quotient of $E_n^r$. It is a semisimple algebra (see Example \ref{Ex1}), so that by Corollary \ref{sequence of surjections} we get

\begin{thm} \label {ss quotients}
 For all $n$ and all $r$ the algebras $\overline E_m^r$, $m=1, 2, \cdots , n$ are semisimple quotient algebras of $E_n^r$.
\end{thm}

\begin{remarkcounter} \label{p-term is 0} \begin{enumerate}
\item [a)] We have $\overline E_m^r = 0$ for all $r$ when $m \geq p$. This is clear for $m > p$, because then $A_m(p) = \emptyset$. If $m = p$, we have that $\epsilon_1$ belongs to the upper wall of $A_m(p)$, see Remark \ref{A for type A}. Hence, $V_p$ is negligible and therefore so are also all tensor powers $V_p^{\otimes r}$. This means that $V_p^{\underline \otimes r} = 0$ for all $r$. 
\item [b)] We do not have surjections  $\overline E_m^r \to \overline E_{m-1}^r$  analogous to the ones we found in Corollary \ref{sequence of surjections}. In fact,  the alcove $A_p(m)$ become larger the smaller $m$ we consider. This means that the algebras $\overline E_m^r$ grow in size when $m$ decreases.
\end{enumerate}
\end{remarkcounter}

\subsection{A class of simple modules for symmetric groups} \label{class of simple}

 The group algebra $kS_r$ of the symmetric group on $r$ letters is isomorphic to the algebra $E_n^r$ for all $n \geq r$ see e.g. \cite[3.1]{CL}. Hence, by Theorem \ref{ss quotients} $kS_r$ has the following list of semisimple quotients: $\overline E_1^r, \overline E_2^r, \cdots , \overline E_r^r$. As observed in Remark \ref{p-term is 0} we have $\overline E_n^r= 0$, if $n \geq p$. On the other hand, $kS_r$ is itself semisimple if $p >r$, and its representation theory coincides with the well-known theory in characteristic $0$. So we shall assume in the following that 
$p \leq r$.

In the special case $n=1$ we have  $V_1^{\otimes r}= k_{r \epsilon_1}$, $r \in \Z$ and these modules together with their duals are the indecomposable tilting modules (as well as the simple modules) for $G_1$. The fusion category for $G_1$  coincides with the full category of finite-dimensional $G_1$-modules. We identify the trivial $1$ line partion of $r>0$ with the element $r \epsilon_1$ in $A_1(p)$. Clearly, we have  
$\overline E_1^r =  E_1^r = \End_{G_1}(k_{r \epsilon_1}) = k$ for all $r$.

We shall explore the simple modules for $kS_r$ arising from the above quotients $\overline E_m^r$. Note that we just observed that the first algebra $\overline E_1^r$ equals $k$.

Consider the remaining quotients  $\overline E_m^r$, $m= 2, \cdots , p-1$ of $kS_r$. We shall describe the simple modules for $kS_r$ arising from these. Recall that the simple modules for $kS_r$ are indexed by the $p$-regular partitions of $r$, i.e. partitions of $r$ with no $p$ rows having the same length. If $\lambda$ is such a partion, we denote the corresponding simple module for $kS_r$ by $D_r(\lambda)$. 

Set  $\Lambda^r$ equal to the set of partitions of $r$. This is the weight set for the cellular algebra $E_n^r$ whenever $n \geq r$. Define 
$$\overline \Lambda^r(p) = \{ (\lambda_1, \lambda_2, \cdots ,\lambda_m) \in \Lambda^r | \lambda \in A_m(p) \text { for some } m < p \}.$$
So $\overline \Lambda^r(p) $ consists of those partitions of $r$ which have at most $m$ non-zero terms and satisfy $\lambda_1 - \lambda_m \leq p - m$.  Clearly, the partions in $\overline \Lambda^r(p)$ are all $p$-regular. We shall now derive an algorithm which  determines  the dimensions of the simple modules  $D_r(\lambda)$ when $\lambda \in \overline \Lambda^r(p)$. 

We have the following Pieri-type branching formula, which is proved e.g. in \cite[(3.7)]{AS}.
\begin{prop} \label{inductive formula}
Let $m \geq 1$ and suppose $\lambda \in A_m(p)$. Then 
$$(V_m^{\otimes r}: T(\lambda)) = \sum_{i: \lambda - \epsilon_i \in \Lambda^{r-1} \cap A_m(p)} (V_m^{\otimes (r-1)}: T(\lambda - \epsilon_i)).$$
\end{prop} 

\begin{lem} \label {the p-1 algebra}
Suppose $1 \leq r = a (p-1) + b$ where $0 \leq b < p-1$.  Then $V_{p-1}^{\underline \otimes r} = T(a\omega_{p-1} + \omega_b)$. Hence,  $\overline E_{p-1}^r = k$.
\end{lem}

\begin{proof} The lemma is clearly true when $r =1$ where $V_{p-1} = T(\omega_1) = L(\omega_1)$. Observe that (with the notation in the lemma) $a\omega_{p-1} + \omega_b$ is the unique element in $\Lambda^r \cap A_{p-1}(p)$. Hence, for $r>1$ the statement follows by induction from Proposition \ref{inductive formula}.
\end{proof}

\begin{thm} \label{main symm}
Let $r > 0$ and suppose $\lambda \in \overline \Lambda^r(p)$. Then the dimension of the simple $kS_r$-module $D_r(\lambda)$ is recursively determined by 
$$ \dim D_r(\lambda) = \sum_{i: \lambda - \epsilon_i \in \overline \Lambda^{(r-1)}(p)} \dim D_{r-1}(\lambda - \epsilon_i).$$
\end{thm}

\begin{proof} For any partition $\mu$ of $r$ the corresponding Specht module for $kS_r$ identifies with the cell module $C_r(\lambda) =\Hom_{G_r}(\Delta_r(\mu), V_r^{\otimes r})$ for $E_r^r \simeq kS_r$. Now Proposition \ref{restriction of Specht} shows that, if $\mu$ has at most $m$ terms, then we have $C_r(\mu) \simeq C_m(\mu)$. The surjection $V_m^{\otimes r}$ onto the fusion summand $V_m^{\underline \otimes r}$ then gives a surjection of $C_m(\mu)$ onto the cell module $\overline C_m(\mu)) = \Hom_{G_m}(\Delta_m(\mu), V_m^{\underline \otimes r})$ for the semisimple quotient algebra $\overline E_m^r$ of $kS_r$.  This latter module is only non-zero if $m < p$ and $\mu \in A_m(p)$. So if $\mu = \lambda$ with $\lambda$  as in the theorem, we see that $D_r(\lambda) = \overline C_m(\lambda)$. The theorem therefore follows from Proposition \ref{inductive formula} by observing that $\dim \overline C_m(\lambda) = (V_m^{\otimes r} : T(\lambda))$, cf. (\refeq{dim simple/tilting}).

\end{proof}

\begin{examplecounter}
Consider the case $p = 3$. Here we have
$$ \overline \Lambda ^r(3) = \begin{cases} \{(1)\} \text { if } r = 1,  \\ \{(r), ( (r+1)/2,  (r-1)/2)\} \text { if } r \geq 3  \text { is odd,} \\ \{ (r), ( r/2, r/2)\} \text { if } r \geq 2 \text { is even.} \end{cases}$$
The trivial partition $(r)$ of $r$ corresponds to the trivial simple module $D_r((r)) = k$ (this is true for all primes). For the unique $2$-parts partition $ \lambda$ in $\overline \Lambda ^r$ we get from Theorem \ref{main symm}

$$ \dim D_r(\lambda) = \begin{cases} \dim D_{r-1}(\lambda - \epsilon_1) \text { if $r$ is odd,} \\ \dim D_{r-1}(\lambda - \epsilon_2) \text { if $r$ is even.} \end{cases}$$
Hence we find $\dim D_r(\lambda) = 1$ for all $r$. This is of course also an immediate consequence of the fact that in this case $\overline E_2^r =k$, see Lemma \ref{the p-1 algebra}. Note that $\overline E_2^r$ is the modular Jones algebra appearing in \cite[Section 7] {A17} and it was observed there as well that the Jones algebras are all trivial in characteristic $3$.

\end{examplecounter}

\begin{examplecounter}
Consider now $p = 5$. Then for $r \geq 5$ we have exactly two partitions $\lambda^1(r)$ and $\lambda^2(r)$ of $r$ having $2$ non-zero parts, which belong to $A_2(5)$. Likewise, there are exactly $2$ partitions $\mu^1(r)$ and $\mu^2(r)$ of $r$ with $3$ non-zero parts, which belong to $A_3(p)$. Finally, there is a unique partition $\nu(r)$ of $r$ with $4$ non-zero parts which belongs to $A_4(p)$. To be precise we have
$$ \lambda^1(r) = ((r+2)/2, (r-2)/2) \text { and } \lambda^2(r) = (r/2, r/2), \text { if $r$ is even;}$$
whereas 
$$ \lambda^1(r) = ((r+3)/2, (r-3)/2) \text { and } \lambda^2(r) = ((r+1)/2,(r-1)/2), \text { if $r$ is odd.}$$
We leave to the reader to work out the formulas for $\mu^1(r), \mu^2(r)$. The expression for $\nu(r)$ is given in Lemma \ref{the p-1 algebra}. 

So $\overline \Lambda^r (p) 
= \{(r), \lambda^1(r), \lambda^2(r), \mu^1(r), \mu^2(r), \nu(r) \}$. We choose the enumeration such that $\lambda^1(r) > \lambda^2(r)$ (in the dominance order) and likewise $\mu^1(r) > \mu^2(r)$. For each of these 6 weights we can easily compute the dimension of the corresponding simple $kS_r$-modules via Theorem \ref{main symm}. In Table 1 we have illustrated the results for $ r \leq 10$. In this table the numbers in row $r$ (listed in the above order) are the dimensions of these $6$ simple $kS_r$-modules. When $r$ is small 
some weights are repeated, e.g.  for $r = 3$ we have $(3) = \lambda^1(3)$, $\lambda^2(3) = \mu^1(3)$ and $\mu^2(3) = \nu(3)$.

\vskip .5 cm 
\centerline {
 {\it Table  1.  Dimensions of simple modules for $kS_r$ when $p= 5$}}

\vskip .5cm
\centerline{
\begin{tabular} { r| c | c c |c c | cc } 
   r &(r) & $\lambda^1(r)$ & $\lambda^2(r)$ & $\mu^1(r)$ & $\mu^2(r)$& $\nu(r)$   \\   
\hline 
        1 & 1 & &1 &   1& & 1\\ 
  2 & 1 & 1 & 1 & 1 & 1 & 1 &    \\ 
  3 & 1& 1 & 2& 2 & 1 & 1  \\
  4 & 1 &3 & 2 & 2 & 3 & 1\\
  5 &1 & 3 & 5 & 3 & 5 &1\\
  6 & 1 & 8& 5& 8 &5  &1  \\
  7 &1 &8 &13 &8 & 13 & 1\\
  8 & 1& 21&13 &13 & 21 & 1\\
 9 & 1 &21&34&34 & 21 & 1\\
10 & 1 &55&34&34 & 55 & 1\\
\end{tabular}}
\vskip .5 cm

The table can easily be extended using the following formulas. Set $a^j(r) = \dim D_r(\lambda^j(r)), \; j=1, 2$. Then Theorem \ref{main symm} gives  $a^1(1) = 0, a^2(1) = 1 = a^2(2)$ and the following recursion rules
$$ a^1(2r+1) = a^1(2r) = a^1(2r-1) + a^2(2r-1); \; a^2(2r+2) = a^2(2r+1) = a^1(2r) + a^2(2r)$$
for $r \geq 1$.
Another way of phrasing this is  that $a^2(1), a^1(2), a^2(3), a^1(4), a^2(5), a^1(6), \cdots $ is the Fibonacci sequence. The first equations above then determine the remaining numbers $a^j(r)$. 

Again we leave it to the reader to find the similar recursion for the dimension of the simple modules corresponding to the $\mu^j(r)$'s. Apart from the fact, that the recursion rules coincide, we see no obvious representation theoretic explanation for the ``symmetry" between the numbers involving $\lambda$'s and those involving $\mu$'s.

\end{examplecounter}
\vskip 1cm

\section{Semisimple quotients of the Brauer algebras} \label{Brauer}

In this section we shall apply our results from Section 2 to the symplectic and orthogonal groups. This will allow us via Schur--Weyl duality for these groups to obtain certain semisimple quotients of the Brauer algebras over $k$ and to give an algorithm for finding the dimensions of the corresponding simple modules.

The Brauer algebra $\mathcal B_r(\delta)$ with parameter $\delta \in k$ may be defined via generators and relations, see e.g. \cite[Introduction]{DDH}. Alternatively, we may consider it first as 
%\begin{defn} Let $ r \in \Z_{> 0}$ and $\delta \in k$. The Brauer algebra $\mathcal B_r(\delta)$ is the k-algebra generated by $\{\sigma_i, u_i | i = 1, \cdots , r-1 \}$ subject to the relations 
%\begin{enumerate}
%\item {$ \sigma_i^2 = 1 \text { for all } i;$}
%\item $ \sigma_i \sigma_j = \sigma_j \sigma_i \text { when } |i-j| > 1;$
%\item $ \sigma_i \sigma_j \sigma_i= \sigma_j \sigma_i \sigma_j \text { when } |i-j| = 1; $ 
%\item $ u_i^2 = \delta u_i \text { for all } i;$
%\item $ u_iu_j = u_ju_i \text { when } |i-j| > 1;$ 
%\item $u_iu_ju_i = u_ju_iu_j  \text { when } |i-j| = 1;$
%\item $ \sigma_i u_j u_i = \sigma_j u_i \text { when } |i-j| = 1, \sigma_i u_i = u_i = u_i \sigma_i \text { for all i.} $
%\end{enumerate}
%\end{defn}
just a vector space over $k$ with basis consisting of the so-called Brauer diagrams with $r$ strands. Then one defines a multiplication of two such diagrams by stacking the second diagram on top of the first, see e.g. \cite{B} or \cite[Section 4]{GL}. This gives an algebra structure on $\mathcal B_r(\delta)$.

We have Brauer algebras for an arbitrary parameter $\delta \in k$. However, the ones that are connected with our endomorhism algebras are those where $\delta$ is the image in $k$ of an integer, i.e. where $\delta$ belongs to the prime field $\Fu \subset k$. It follows from the various versions of the Schur--Weyl duality (see below) that in this case $\mathcal B_r(\delta)$ surjects onto the endomorphism algebra of the $r$'th tensor power of the natural modules for appropriate symplectic and orthogonal groups.

\subsection{Quotients arising from the symplectic groups} \label{quotients sp}

We shall use the notation from Section \ref{Sp}. In particular, $V$ will be a $2n$-dimensional symplectic vector space, which we from now on denote $V_n$. We set $G_n = SP(V_n)$ and $E_n^r = \End_{G_n}(V_n^{ \otimes r})$. 

Consider now the fusion summand $V_n^{\underline \otimes r}$ of $V_n^{\otimes r}$ with endomorphism ring $\overline E_n^r = \End_{G_n}(V^{\underline \otimes r})$. Then exactly as in Proposition \ref{surj GL} we obtain:
\begin{prop} \label{quotients Sp} For all $n$ and $r$ the algebra $\overline E_n^r$ is a semisimple quotient of $E_n^r$. 
\end{prop}

Recalling the description of $A_m(p)$ from Section 2.3 and using Remark 3 we see that $\overline E_n^r = 0$ unless $2n \leq p-1$. 
In contrast with the $GL(V)$ case we usually do not have $\overline E_1^r = k$. In fact, $G_1 = SL_2(k)$ and the tensor powers of the natural module for $G_1$ therefore typically have many summands. On the other hand, the top non-zero term is always equal to $k$: 

\begin{prop} 
$\overline E_{(p-1)/2}^r = k$ for all $r$.
\end{prop}
\begin{proof}
In this proof we drop the subscript $ {(p-1)/2}$ on $V$ and $\Delta$. We have that $V \otimes V $ has a $\Delta$-filtration with factors  $\Delta(2 \epsilon_1), \Delta(\epsilon_1 + \epsilon_2)$ and $\Delta(0) = k$. The first two of these have highest weights on the upper wall of $A_{(p-1)/2}(p)$ whereas the highest weight $0$ of the last term belongs to $A_{(p-1)/2}(p)$. It follows that 
$V \underline \otimes V = k$. Hence,
$$V^{\underline \otimes r} = \begin{cases} V \text { if $r$ is odd,} \\ k \text { if $r $ is even.} \end{cases}$$
The claim follows.
\end{proof}

The analogue of Proposition \ref{inductive formula} is

\begin{prop} \label{inductive formula Sp}
Let $m \geq 1$ and suppose $\lambda \in A_m(p)$. Then 
$$(V_m^{\otimes r}: T_m(\lambda)) = \sum_{i: \lambda \pm \epsilon_i \in A_m(p)} (V_m^{\otimes (r-1)}: T_m(\lambda \pm \epsilon_i)).$$
\end{prop} 

\begin{proof} As $\epsilon_1$ is minuscule we have for any $\lambda \in X_n^+$ that the $\Delta$-factors in $\Delta(\lambda) \otimes V_m$ are those with highest weights $\lambda + \mu$ where $\mu$ runs through the weights of $V_m$ (ignoring possible $\mu$'s for which $\lambda + \mu$ belong to the boundary of $X_m^+$). Likewise,  if $\lambda \in A_m(p)$ then the same highest weights all belong to the closure of $A_m(p)$. Hence the fusion product $\Delta(\lambda) \underline \otimes V$ is the direct sum of all $\Delta_m(\lambda + \mu)$ for which $\lambda + \mu \in A_m(p)$. As the possible $\mu$'s are the $\pm \epsilon_i$ (each having multiplicity $1$) we get the formula.
\end{proof} 

Recall now the Schur--Weyl duality theorem for $SP(V)$, see \cite{DDH}.

\begin{thm} \label{Schur-Weyl Sp} There is an action of $\mathcal B_r(-2n)$ on $V_n^{\otimes r}$ which commutes with the action of $G_n$. The corresponding 
homomorphism $\mathcal B_r(-2n) \rightarrow E_n^r$ is surjective for all $n$ and for $n\geq r$ it is an isomorphism.
\end{thm}

The simple modules for $\mathcal B_r(\delta)$ are parametrized by the $p$-regular partitions of $r, r-2, \cdots$, see \cite[Section 4]{GL}, and we shall denote them $D_{\mathcal B_r(\delta)}(\lambda)$.
 This parametrization holds for any $\delta \in k$. However, in this section we only consider the case where $\delta$ is the image in $k$ of a negative even number. We identify $\delta$ with an integer in $[0, p-1]$. 
%For the other cases $\mathcal B_r(\delta)$ is related to supergroups, and for this we refer the reader to \cite{ES}.

Assume $\delta$ is odd and define the following subsets of weights
$$ \overline \Lambda^r(\delta, p) =  (\Lambda^r \cup \Lambda^{r-2} \cup \cdots ) \cap A_{(p-\delta)/2}(p).$$
So if $\delta < p-2$ then $ \overline \Lambda^r(\delta, p)$ consists of partitions $\lambda = (\lambda_1, \lambda_2, \cdots \lambda_{(p-\delta)/2})$ with $|\lambda| = r - 2i$ for some $i \leq r/2$ which satisfy $\lambda_1 + \lambda_2 \leq \delta$. On the other hand,
$\overline \Lambda^r(p-2, p) = \{(r-2i) | r-p+2 \leq 2i \leq r\}$.

Note that all partitions in $\overline \Lambda^r(p)$ are $p$-regular.

\begin{thm} \label{main brauer Sp}
Let $r > 0$ and  consider an odd number $\delta \in [0,p-1]$. Suppose $\lambda \in \overline \Lambda^r(\delta, p)$. Then the dimension of the simple $\mathcal B_r(\delta)$-module $D_{\mathcal B_r(\delta)}(\lambda)$ is recursively determined by 
$$ \dim D_{\mathcal B_r(\delta)}(\lambda) = \sum_{i: \lambda \pm \epsilon_i \in \overline \Lambda^{(r-1)}(\delta, p)} \dim D_{\mathcal B_{r-1}(\delta)}(\lambda \pm \epsilon_i).$$
\end{thm}

\begin{proof} Combining Theorem \ref{Schur-Weyl Sp} with Theorem \ref{quotients Sp} we see that $\overline E_{(p-\delta)/2}^r$ is a semisimple quotient of $B_r(\delta)$. Then the theorem follows from Proposition \ref{inductive formula Sp} by recalling that the dimensions of the simple modules for 
$\overline E_{(p-\delta)/2}^r$ coincide with the tilting multiplicities in $V_{(p-\delta)/2}^{\otimes r}$, see (\refeq{dim simple/tilting}).
\end{proof}

\begin{remarkcounter} If $n \equiv (p-\delta)/2\; (\mo p)$ for some odd number $\delta \in [0,p-1]$, then $-2n \equiv \delta$. Hence, the theorem describes a class of simple modules for $\mathcal B_r(-2n)$ for all such $n$.
\end{remarkcounter}

\begin{examplecounter} \label{brauer p=7}
Consider $p = 7$. Then the relevant $\delta$'s are $5, 3$ and $1$. The weight set $\overline \Lambda^r(5,7)$ contains $3$ elements  (except for $r < 4$ where there are fewer) $\lambda^1(r), \lambda^2(r), \lambda^3(r)$ listed in descending order, namely $(4), (2), (0)$, when $ r$ is even, and $(5), (3), (1)$, when $r$ is odd. Likewise, $\overline \Lambda^r(3,7)$ contains $3$ elements  (except for $r = 1$) $\mu^1(r), \mu^2(r), \mu^3(r)$ listed in descending order, namely $(2,0), (1,1), (0,0)$, when $ r$ is even, and $(3,0), (2,1), (1,0)$, when $r$ is odd. Finally, $\overline \Lambda^r(1,7)$ consists of a unique element $\nu(r)$, namely $\nu(r) = (0,0,0)$, when $r$ is even, and $\nu(r) =(1,0,0)$, when $r$ is odd.

In Table 2 we have listed the dimensions of the simple modules for $\mathcal B_r(\delta) $ for $r \leq 10$. These numbers are computed recursively using Theorem \ref{main brauer Sp}.
\end{examplecounter}

\eject
\centerline {
{ \it Table  2.  Dimensions of simple modules for $\mathcal B_r(\delta)$ when $p= 7$ and $\delta = 5, 3, 1$.
}}
\vskip .5cm
\centerline {
\begin{tabular}{ r| c  c c |c c c |c| c c c c c c c c c c}
        $\delta$ &&5&&&3&&1& \\ \hline
        r & $\lambda^1(r)$ & $\lambda^2(r)$ & $\lambda^3(r)$ & $\mu^1(r)$ & $\mu^2(r)$ & $\mu^3(r)$ & $\nu(r)$    \\   \hline 
     1 &  & & 1  &  & &1 & 1& \\ 
  2 &  & 1 &1  & 1 & 1 & 1 & 1   \\ 
  3 & & 1& 2& 1 & 2 & 3 & 1  \\
  4 & 1 &3 & 2 & 6 & 5 & 3& 1 \\
  5 &1 & 4 & 5 & 6 & 11 &14 & 1\\
  6 & 5 & 9& 5& 31&25 &14 & 1  \\
  7 &5 &14 &14 &31& 56 & 70 & 1\\
  8 & 19& 28&14 &157 & 126 &70 & 1\\
 9 & 19 &47&42&157 & 283 & 353 & 1\\
10 & 66 &89&42&793 & 636 & 353 & 1\\
\end{tabular}}
\vskip 1 cm 
\subsection{Quotients arising from the orthogonal groups}
In this section we consider the orthogonal groups. Again we shall see that the very same methods as we used for general linear groups in Section 4 apply in this case.

We shall use the notation from Section \ref{O}. In particular, $V$ will be a vector space with a non-degenerate symmetric bilinear form. If $\dim V$ is odd, we write $\dim V = 2n +1$ and set $V_n = V$ and $G_n = O(V)$. Likewise, if $\dim V$ is even, we write $\dim V = 2n$ and set $V_n = V$ and $G_n = O(V_n)$. In both cases we denote by $E_n^r$ the endomorphism algebra $ \End_{G_n}(V_n^{ \otimes r})$ and by $\overline E_n^r$ the algebra $\End_{G_n}(V_n^{\underline \otimes r})$.

As in the general linear and the symplectic case we have: 
\begin{prop} \label{quotients O} For all $n$ and $r$ the algebra $\overline E_n^r$ is a semisimple quotient of $E_n^r$. 
\end{prop}

Recalling the description of $A_m(p)$ from Section \ref{O} we observe:
\begin{remarkcounter} \label{orthogonal barE} \begin{enumerate} 
\item [a)]By Remarks \ref{A for type B} and \ref{A for type D} we have $\overline E_n^r$ = 0 unless $2n < p-2$ in the odd case, respectively $2n < p+2$ in the even case.
\item [b)] In the even case we get $\overline E_{(p+1)/2}^r = k$ for all $r$ using the same argument as in the symplectic case. On the other hand, this argument does not apply to the odd case, where in fact the highest term $\overline E_{(p-3)/2}^r$ is usually not  $k$ (this is illustrated in Example \ref{brauer2 p=7} below).  

\end{enumerate}
\end{remarkcounter}

The Schur--Weyl duality for orthogonal groups \cite[Theorem 1.2]{DH} says in particular: 

\begin{thm} \label{Schur-Weyl O} Set $\delta = \dim V_n$. There is an action of $\mathcal B_r(\delta)$ on $V_n^{\otimes r}$ which commutes with the action of $G_n$. The corresponding 
homomorphism $\mathcal B_r(\delta) \rightarrow E_n^r$ is surjective for all $n$.
\end{thm}

\begin{remarkcounter}
The Schur--Weyl duality for orthogonal groups gives rise to isomorphisms for large enough $n$, see e.g.  \cite[Section 3.4]{AST2}. We shall not need this here.
\end{remarkcounter}

We now divide our discussion into the odd and even cases.

\subsubsection{Type B}
In the odd case where $G_n$ has type $B_n$ our methods lead to the higher Jones quotient $\overline E_m^r$ of $\mathcal B_r(2m+1)$ for  $1 \leq m \leq (p-3)/2$. Noting that the Brauer algebras in question are those with an odd $\delta $ lying between $3$ and $p-2$ which we have already dealt with in Section \ref{quotients sp},  we shall leave most details to the reader. However, we do want to point out that the inductive formula for 
the dimensions of the simple modules for $\overline E_n^r$ in this case is more complicated than in the symplectic case. The reason is that for type $B$ the highest weight for the natural module is not minuscule. This means that instead of the direct analogue of Proposition \ref{inductive formula} we need to use the following general formula (with notation as in Section 2).

\begin{thm} (\cite[Equation 3.20(1)]{AP}). 
 Let $G$ be an arbitrary reductive group over $k$ and suppose $Q$ is a tilting module for $G$. If $\lambda$ is a weight belonging to the bottom dominant alcove $A(p)$, then
$$ (Q:T(\lambda)) = \sum_{w} (-1)^{\ell (w)} (Q:\Delta(w \cdot \lambda)),$$
where the sum runs over those $w \in W_p$ for which $w \cdot \lambda \in X^+$.
\end{thm}

\begin{examplecounter} \label{brauer2 p=7} Consider $p =7$. Then type $B$ leads to higher Jones algebras of $\mathcal B_r(3)$ and $\mathcal B_r(5)$. The reader may check that the recursively derived dimensions for the class of simple modules in these cases match (with proper identification of the labeling) with those listed in Table 2.  Note in particular that to get those for $\mathcal B_r(3)$ we need to decompose $V_1^{\underline \otimes r}$ into simple modules for $G_1$. The Lie algebra for $G_1$ is $\mathfrak{sl}_2$ and the natural $G_1$-module $V_1$ identifies with the simple $3$-dimensional $SL_2$-module. 
\end{examplecounter}

\subsubsection{Type D}
In the even case $G_n$ has type $D$. The module $V_n$ equals $\Delta_n(\epsilon_1)$ and its highest weight $\epsilon_1$ is minuscule. This means that we have

\begin{prop} \label{inductive formula O}
Let $n \geq 1$ and suppose $\lambda \in A_n(p)$. Then 
$$(V_n^{\otimes r}: T_n(\lambda)) = \sum_{i: \lambda \pm \epsilon_i \in A_n(p)} (V_n^{\otimes (r-1)}: T_n(\lambda \pm \epsilon_i)).$$
\end{prop} 

\begin{proof} Completely analogous to the proof of Proposition \ref{inductive formula Sp}.
\end{proof} 

Assume now $\delta \in [2, p+1]$ is even and define the following subsets of weights
$$ \overline \Lambda^r(2, p) = \{(r-2i) | 0 \leq r-2i \leq p-2\},$$
$$  \overline \Lambda^r(4, p) =\{(\lambda_1, \lambda_2) \in X_2^+ | (\lambda_1,|\lambda_2|) \in \Lambda^{r-2i} \text { for some $i$ with }  0 \leq r-2i \leq p-2\}, $$ 
and for $\delta > 4$ 
$$ \overline \Lambda^r(\delta, p)  = \{ (\lambda_1, \lambda_2, \cdots , \lambda_{\delta/2}) \in X_{\delta/2}^+ | (\lambda_1, \cdots , |\lambda_{\delta/2}|) \in \Lambda^{r-2i} \text { for some } i \leq r/2 \text { and } \lambda_1 + \lambda_2 \leq p-\delta + 2\}. $$

\vskip .3 cm

\begin{thm} \label{main brauer O}
Let $r > 0$ and  consider an even number $\delta \in [0,p+1]$. Suppose $\lambda \in \overline \Lambda^r(\delta, p)$. Then the dimension of the simple $\mathcal B_r(\delta)$-module $D_{\mathcal B_r(\delta)}(\lambda)$ is recursively determined by 
$$ \dim D_{\mathcal B_r(\delta)}(\lambda) = \sum_{i: \lambda \pm \epsilon_i \in \overline \Lambda^{(r-1)}(\delta, p)} \dim D_{\mathcal B_{r-1}(\delta)}(\lambda \pm \epsilon_i).$$
\end{thm}

\begin{proof} Combining Theorem \ref{Schur-Weyl O} with Theorem \ref{quotients O} we see that $\overline E_{(p-\delta)/2}^r$ is a semisimple quotient of $B_r(\delta)$. Then the theorem follows from Proposition \ref{inductive formula O} by recalling that the dimensions of the simple modules for
 $\overline E_{\delta/2}^r$ coincide with the tilting multiplicities in $V_{\delta/2}^{\otimes r}$, see (\refeq{dim simple/tilting}).
\end{proof}

\begin{remarkcounter} If $n \equiv \delta/2 \; (\mo \; p)$ for some even number $\delta \in [2,p+1]$, then $2n \equiv \delta \; (\mo \;p)$. Hence, the theorem describes a class of simple modules for $\mathcal B_r(2n)$ for all such $n$.
\end{remarkcounter}

\begin{examplecounter} Consider $p = 7$. Then the relevant $\delta$'s are $2, 4, 6, 8$. By Remark \ref{orthogonal barE}b we have that the higher Jones quotient algebra for $\mathcal B_r(8)$ is the trivial algebra $k$ (alternatively, observe that $\mathcal B_r(8) = \mathcal B_r(1)$ which we dealt with in Example \ref{brauer p=7}).  At the other extreme the (higher) Jones quotient of $\mathcal B_r(2)$ is also a quotient of the Temperley--Lieb algebra $TL_r(2)$.  This case is dealt with in \cite[Proposition 6.4]{A17}. So here we only consider the two remaining cases $\delta = 4$ and $\delta = 6$. We have
$$ \overline \Lambda^1(4,7) = \{(1,0)\},$$
$$ \overline \Lambda^2(4,7) = \{(2,0), (1,1), (1,-1), (0,0)\},$$
$$ \overline \Lambda^3(4,7) = \{(3,0), (2,1), (2,-1), (1,0)\},$$
$$ \overline \Lambda^r(4,7) = \begin{cases} \{(4,0), (2,2), (2,-2), (3,1) (3,-1), (2,0), (1,1), (1,-1), (0,0)\} \text { if $r \geq 4$ is even,} \\  \{(5,0), (3,2), (3,-2), (4,1), (4,-1), (3,0), (2,1), (2,-1), (1,0)\} \text { if $r \geq 5$ is odd.} \end{cases}$$
In Table 3 we have denoted these weights $\lambda^1(r), \cdots , \lambda^9(r)$.

Likewise, we have
$$ \overline \Lambda^1(6,7) = \{(1,0,0))\},$$
$$ \overline \Lambda^2(6,7) = \{(2,0,0), (1,1,0), (0,0,0)\},$$
$$ \overline \Lambda^r(6,7) = \begin{cases} \{(3,0,0), (2,1,0), (1,1,1), (1,1,-1), (1,0,0)\} \text { if $r \geq 3$ is odd.} \\  \{(2,1,1)), (2,1,-1)), (2,0,0), (1,1,0),(0,0,0)\} \text { if $r \geq 4$ is even.} \\  \end{cases}$$
In Table 3 we have denoted these weights $\mu^1(r), \cdots , \mu^5(r)$. In this table we have then listed the dimensions (computed via the algorithm in Theorem \ref{main brauer O}) for the simple modules for $\mathcal B_r(4)$, respectively $\mathcal B_r(6)$  corresponding to these sets of weights for $r \leq 10$.

\eject
\centerline {{ \it Table  3.  Dimensions of simple modules for $\mathcal B_r(\delta)$ when $p= 7$ and $\delta = 4 $ and $6$.}}
\vskip .5cm

\noindent
\begin{tabular}{ r| c  c c c c c c c c | c c c c c c}
         &&& &&$\delta =4$&&&&&&&$\delta = 6$&& \\ \hline
        r & $\lambda^1(r)$ & $\lambda^2(r)$ & $\lambda^3(r)$&$\lambda^4(r)$&$\lambda^5(r)$&$\lambda^6(r)$&$\lambda^7(r)$&$\lambda^8(r)$&$\lambda^9(r)$ & $\mu^1(r)$& $\mu^2(r)$ & $\mu^3(r)$ &$\mu^4(r)$ & $\mu^5(r)$    \\   \hline 
  
     1 &  & &  & &&&&&1 & & &&& 1& \\ 
  2 &&&&&&1 & 1  & 1 & 1 & &&1& 1 & 1   \\ 
  3 &&&&&&1& 2&2 & 4 & 1&2&1&1& 3  \\
  4 & 1 &2 & 2 & 3 & 3 & 9&6&6&4& 2 & 3 & 6 & 7 & 3 \\
  5 &1 & 5 & 5 & 4 & 4& 16 &20 & 20 &25 &6 & 18 &9&10&16 \\
  6 & 25& 25& 25& 45 & 45&81 & 45 & 45 & 25 &27&28&40&53 & 16&  \\
  7 &25 &70 &70 &70& 70 & 196& 196 & 196 & 196 &40&148&80 &81 &109 \\
  8 & 361,& 266 & 266 & 532 & 532 & 784 & 392 & 392 & 196 &228 &229 &297 &418  &109 \\
 9 & 361&798& 798 & 893 & 893 & 2209 & 1974 & 1974 & 1764  &297&1172&646&647&824 \\
10 & 4356 & 2772 & 2772 & 5874 & 5874 & 7921 & 3738 & 3738 & 1764 &1828&1829&2293&3289&824 \\
\end{tabular}
\vskip .5 cm
Together with Example \ref{brauer p=7} this example give a class of simple modules for Brauer algebras with parameter $\delta$ equal to any non-zero element of $\mathbb {F}_7$. 

\end{examplecounter}

Note that in the above example we were in type $D_1 = A_1$, $D_2 = A_1 \times A_1$ or $D_3 = A_3$ and we could have deduced the results from the Type A case treated in Section 4. We shall now give another example illustrating  type $ D_n$ computations  with $n > 3$. 
\begin{examplecounter}
Consider $p = 11$ and take $\delta = 10$. Then $\mathcal B_r(10)$ has the higher Jones quotient $\overline E_5^r$. If $r \geq 5$ the weight set $\overline \Lambda^r(10, 11)$ contains $7$ elements, namely 
$$ \{(1,1,1,1,-1), (2,1,1,1,0), (1,1,1,1,1), (3, 0,0,0,0), (2,1,0,0,0), (1,1,1,0,0), (1,0,0,0,0)\}$$
when $r$ is odd, and 
$$ \{(2,1,1,1, -1), (2,1,1,1,1), (2,1,1,0,0), (1,1,1,1,0), (2,0,0,0,0), (1,1,0,0,0), (0,0,0,0,0)\}$$
when $r$ is even.

If $r \in \{1,2,3,4\}$, the set  $\overline \Lambda^r(10, 11)$ consists of the last, the $3$ last, the $4$ last, and the $5$ last elements, respectively, in the above lists. 

In Table 4 we have listed the dimensions of the corresponding simple modules for $\mathcal B_r(10)$ for $r \leq 10$ using Theorem \ref{main brauer O}. We have denoted the above $7$ weights $\lambda^1(r), \cdots , \lambda^7(r)$ (in the given order).
\end{examplecounter}

\eject
\centerline{
{ \it Table  4.  Dimensions of simple modules for $\mathcal B_r(10)$ when $p= 11$.
}}
\vskip .5cm
\centerline {
 \begin{tabular}{ r| c  c c  c c c c|ccc}
   r & $\lambda^1(r)$ & $\lambda^2(r)$ & $\lambda^3(r)$ & $\lambda^4(r)$ & $\lambda^5(r)$ & $\lambda^6(r)$ &$\lambda^7(r)$    \\        \hline 
  1&  & &  & & & & 1& \\ 
  2 &  &&  &  & 1 & 1 & 1 &  \\ 
  3 && & & 1& 2 & 1 & 3 &  \\
  4 &  &&3 & 1 & 6 & 6 & 3&  \\
  5 &1&4 & 1 & 6 & 15 & 10 &15 & \\
  6 & 5&5 & 29& 16& 36&40 &15 &   \\
  7 &21 &55 &21 &36 &105& 85 & 91 & \\
  8 &76 & 76& 245&97 &232 & 281 &91 & \\
 9 &173& 494 &173&232&568 & 623 & 604& \\
10 &667& 667 &1685&840&1404 & 1795 & 604 & \\
\end{tabular}}
\vskip 1 cm

\section{Quantum Groups}

In the remaining sections $k$ will denote an arbitrary field. 

Let $\mathfrak g$ denote a simple complex Lie algebra. Then there is a quantum group $U_q = U_q(\mathfrak g)$  (a quantized enveloping algebra over $k$) associated with $\mathfrak g$. We shall be interested in the case where the quantum parameter $q$ is a root of unity in $k$ and we want to emphazise that the quantum group we are dealing with is the Lusztig version defined via $q$-divided powers, see e.g. \cite[Section 0]{APW}. This means that we start with the ``generic" quantum group $U_v = U_v(\mathfrak g)$ over $\Q(v)$ where $v$ is an indeterminate. Then we consider the $\Z[v,v^{-1}]$-subalgebra $U_{\Z[v,v^{-1}]}$ of $U_v$ generated by the quantum divided powers of the generators for $U_v$. When $q \in k\setminus 0$ we make $k$ into an $\Z[v,v^{-1}]$-algebra by specializing $v$ to $q$ and define $U_q$ as $U_q = U_{\Z[v,v^{-1}]} \otimes_{\Z[v,v^{-1}]} k$. This construction, of course, makes sense for arbitrary $q$, but if $q$ is not a root of unity all finite-dimensional $U_q$-modules are semisimple and our results are trivial. So in the following we always assume that $q$ is a root of unity and we denote by $\ell$ the order of $q$. When $\ell \in \{2, 3, 4, 6\}$ the (quantum) higher Jones algebras we introduce turn out to be trivial ($0$ or $k$) for all $\mathfrak g$ so we ignore these cases. We set $\ell' = \ord(q^2)$, i.e. $\ell' = \ell$, if $\ell$ is odd, and $\ell' = \ell/2$, if $\ell$ is even.

In this section we shall very briefly recall some of the key facts about $U_q$ and its representations relevant for our purposes. As the representation theory for $U_q$ is in many ways similar to the modular representation theory for $G$ that we have been dealing with in the previous sections, we shall leave most details to the reader. However, we want to emphazise one difference: if the root system associated with $\mathfrak g$ has two different root lengths then the case of even $\ell$ is quite different from the odd case (the affine Weyl groups in question are not the same). This phenomenon is illustrated in \cite[Section 6]{AS} where the fusion categories for type $B$ as well as the corresponding fusion rules visibly depend on the parity of $\ell$. The difference will also be apparent in Section 6.3 below where for instance the descriptions of the bottom dominant alcoves in the type $C$ case considered there depend on the parity of $\ell$.

Again, we start out with the general case and then specialize first to the general linear quantum groups, and then to the symplectic quantum groups. We omit treating the case of quantum groups corresponding to the orthogonal Lie algebras, because of the lack of a general version of Schur--Weyl duality in that case.

\subsection{Representation theory for Quantum Groups}
We have a triangular decomposition $U_q = U_q^- U_q^0U_q^+$ of $U_q$. If $n$ denotes the rank of $\mathfrak g$, then we set $X = \Z^n$ and identify each $\lambda \in X$ with a character of $U_q^0$ (see e.g. \cite[Lemma 1.1]{APW}).  These characters extend to $B_q = U_q^-U_q^0$ giving us the $1$-dimensional $B_q$-modules $k_\lambda, \lambda \in X$. As in Section 1.1 we denote by $R$ the root system for $\mathfrak g$  and consider $R$ as a subset of $X$. The set $S$ of simple roots corresponds to the generators of $U_q^+$ and we define the dominant cone $X^+ \subset X$ as before. The Weyl group $W$ is still the group generated by the reflections $s_{\alpha}$ with $\alpha \in S$.

Define the bottom dominant alcove in $X^+$ by
$$ A(\ell) = \begin{cases} \{\lambda \in X^+ | \langle \lambda + \rho, \alpha_0^{\vee} \rangle < \ell \} \text { if $\ell$ is odd,} \\ \{\lambda \in X^+ | \langle \lambda + \rho, \beta_0^{\vee} \rangle < \ell' \} \text { if $\ell$ is even.} \end{cases}$$
Here $\alpha_0$ is the highest short root and $\beta_0$ is the highest long root.

The affine Weyl group $W_\ell$ for $U_q$ is then the group generated by the reflections in the walls of $A(\ell)$. Note that, when $\ell$ is odd, $W_\ell$ is the affine Weyl group (scaled by $\ell$) associated with the dual root system of $R$, whereas if $\ell$ is even, $W_\ell$ is the affine Weyl group (scaled by $\ell'$) for $R$, cf. \cite[Section 3.17]{AP}.

Suppose $\lambda \in X^+$. Then we have modules $\Delta_q(\lambda), \nabla_q(\lambda), L_q(\lambda)$ and $T_q(\lambda)$ completely analogous to the $G$-modules in Section 2 with the same notation without the index $q$. 

The quantum linkage principle (see \cite{A03}) implies that if $L_q(\mu)$ is a composition factor of $\Delta_q(\lambda)$, then $\mu$ is strongly linked (by reflections from $W_\ell$) to $\lambda$. Likewise, if $\Delta_q(\mu)$ occurs in a Weyl filtration of $T_q(\lambda)$, then $\mu$ is strongly linked to $\lambda $.

The quantum linkage principle then gives the identities
$$ \Delta_q(\lambda) = \Delta_q(\lambda) = L_q(\lambda) = T_q(\lambda)  \text { for all } \lambda \in A(\ell),$$
which will be crucial for us in the following.

Suppose $Q$ is a general tilting module for $U_q$. Imitating the definitions in Section \ref{Fusion} we define the fusion summand and the negligible summand of $Q$ as follows
$$  Q^{\F}= \bigoplus_{\lambda  \in A(\ell)} T_q(\lambda)^{(Q:T_q(\lambda))} \text { and }  Q^{\Ne} = \bigoplus_{\lambda  \in X^+\setminus A(\ell)}T_q(\lambda)^{(Q:T_q(\lambda))}$$. 

The exact same arguments as in the modular case then give us the quantum analogue of Theorem \ref{fusion-quotient}
\begin{thm} \label{q-fusion-quotient}
Let $Q$ be an arbitrary tilting module for $U_q$. Then the natural map $\phi: \End_{U_q}(Q) \rightarrow \End_{U_q}(Q^{\F})$ is a surjective algebra homomorphism. The kernel of $\phi$ equals 
%$$\{h \in \End_{U_q}(Q) | \Tr_q(i_\lambda \circ h \circ \pi_\lambda) = 0 \text { for all homomorphisms } i_\lambda: T_q(\lambda) \rightarrow Q, \; \pi_\lambda: Q \rightarrow T_q(\lambda),\; \lambda \in X^+\}.$$
$$\{h \in \End_{U_q}(Q) | \Tr_q(i_\lambda \circ h \circ \pi_\lambda) = 0 \text { for all } i_\lambda \in \Hom_{U_q}( T_q(\lambda), Q), \; \pi_\lambda \in \Hom_{U_q}(Q, T_q(\lambda)),\; \lambda \in X^+\}.$$
\end{thm}

We also have a quantum fusion category (still denoted $\mathcal F$) and a fusion tensor product $\underline \otimes$ on it, see \cite[Section 4]{A92}. This leads to an analogue of Corollary 2.4.

\begin{cor} \label{q-fusion} Let $T$ be an arbitrary tilting module for $U_q$. Then for any $r \in \Z_{\geq 0}$  the natural homomorphism $\End_{U_q}(T^{\otimes r}) \rightarrow \End_{U_q}(T^{\underline \otimes r})$ is surjective.
\end{cor}

All of the above easily adapts to the case, where we replace the simple Lie algebra $\mathfrak g$ by the general linear Lie algebra $\mathfrak {gl}_n$ and we shall explore this case further in the next section.

Finally, the cellular algebra theory recalled in Section \ref{cellular} carries over verbatim. Alternatively, use the quantum framework from \cite[Section 5]{AST1} directly.

\subsection{The General Linear Quantum Group} \label{general linear q-group}

Let $n \geq 1$ and consider the general linear Lie algebra $\mathfrak {gl}_n$. The generic quantum group over $\Q(v)$ associated to $\mathfrak {gl_n}$ has a triangular decomposition in which the $0$ part identifies with a Laurent polynomial algebra $\Q(v)[K_1^{\pm 1}, \cdots , K_n^{\pm 1}]$. If $\lambda = (\lambda_1, \cdots , \lambda_n) \in X_n = \Z^n$ then $\lambda$ defines a character of this algebra which sends 
$K_i$ into $v^{\lambda_i}$. In particular, the element $\epsilon_i  \in X_n$ with $1$ as its $i$-th entry and $0$'s elsewhere defines the character which sends $K_i$ to $v$ and all other $K_j$'s to $1$. We then  have $\lambda = \sum_i \lambda_i \epsilon_i$.

Set $U_{q,n}$ equal to the quantum group for $\mathfrak {gl}_n$ over $k$ with parameter a root of unity $q \in k$. Then we still get for $\lambda \in X_n$ a character of $U_q^0$, see \cite[Section 9]{APW}. If we denote by $V_{q,n}$ the $n$-dimensional vector representation of $U_{q,n}$, then (in analogy with the classical case) $V_{q,n}$ has weights $\epsilon_1, \cdots, \epsilon_n$, all with multiplicity $1$. Moreover, we may (for all $\ell$) identify $V_{q,n}$ with $\Delta_q(\epsilon_1) = \nabla_q(\epsilon_1) = L_q(\epsilon_1) = T_q(\epsilon_1)$. 

The bottom alcove in $X_n$ is now denoted $A_n(\ell)$ and given by 
$$ A_n(\ell) = \{\lambda \in X_n | \lambda_1 \geq \lambda_2 \geq \cdots \geq \lambda_n \text { and } \lambda_1 - \lambda_n \leq \ell'
 -n\}.$$

As noted above,  $V_{q,n}$ is a tilting module. Hence, so are $V_{q,n}^{\otimes r}$ as well as the corresponding fusion summands $V_{q,n}^{\underline \otimes r}$  for all $r \in Z_{\geq 0}$. We set $E_{q,n}^r = \End_{U_{q,n}}(V_{q,n}^{\otimes r})$ and $\overline E_{q,n}^r = \End_{U_{q,n}}(V_{q,n}^{\underline \otimes r})$. These endomorphism algebras are then cellular algebras, and $\overline E_{q,n}^r$ is in fact semisimple (because $V_{q,n}^{\underline \otimes r}$ is a semisimple $U_{q,n}$-module). Moreover, by Corollary \ref{q-fusion} we have:
\begin{equation} \label {E to barE}
\text {The natural homomorphism } E_{q,n}^r \rightarrow \overline E_{q,n}^r \text { is surjective.}
\end{equation}

Arguing as in Section \ref{tensor powers} we also get: 
\begin{equation} \label{surj n>m}
\text {The ``restriction" homomorphisms }E_{q,n}^r \rightarrow E_{q,m}^r \text { are surjective for all } n \geq m \text { and all } r.
\end{equation}

\subsection{Quantum Symplectic Groups} \label{q-sp}

Set now $U_{q,n}$ equal to the quantum group corresponding to the simple Lie algebra $\mathfrak {sp}_{2n}$ of type $C_n$. The vector representation $V_{q,n} = \Delta_q(\epsilon_1)$ is then a tilting module for $U_{q,n}$. As in the corresponding classical case it has weights $\pm \epsilon_1, \cdots , \pm \epsilon_n$ .

The bottom alcove in $X_n$ is now denoted $A_n(\ell)$ and given by 
$$ A_n(\ell) = \begin{cases}  \{\lambda \in X_n | \lambda_1 \geq \lambda_2 \geq \cdots \geq \lambda_n \geq 0 \text { and } \lambda_1 + \lambda_2 \leq \ell - 2n \} \text { if $\ell$ is odd,}\\  \{\lambda \in X_n | \lambda_1 \geq \lambda_2 \geq \cdots \geq \lambda_n \geq 0 \text { and } \lambda_1\leq \ell' -n-1\} \text { if $\ell$ is even.} \end{cases}$$
In both the even and the odd case we have $A_n(\ell) \neq \emptyset$ if and only if  $\ell > 2n$. In the odd case $\epsilon_1$ belongs to $A_n(\ell)$ for $n = 1, 2, \cdots , (\ell -1)/2$, whereas in the even case the same is true for $n=1, 2, \cdots , (\ell -4)/2$.

Again in this case we get (with $E_{q,n}^r = \End_{U_{q,n}}(V_{q,n}^{\otimes r})$ and $\overline E_{q,n}^r = \End_{U_{q,n}}(V_{q,n}^{\underline \otimes r}$)):
\begin{equation} \label{surj E to Ebar qsp} 
\text {The natural homomorphisms } E_{q,n}^r \rightarrow \overline E_{q,n}^r \text { are surjective for all $n, r$.}
\end{equation}

\section{A class of simple modules for the Hecke algebras of symmetric groups}

We continue in this section to assume that $k$ is an arbitrary field and that $q \in k$ is a root of  order $\ell$. 

Let $r$ be a positive integer and denote by $H_r(q)$ the Hecke algebra of the symmetric group $S_r$ with parameter $q\in k$.  
%This is the algebra with generators $T_i$, $i=1,2, ,\cdots , r-1$  and relations 
%\begin{enumerate}
%\item $(T_i-q)(T_i +1) = 0$ for all $i$,
%\item $T_iT_j = T_jT_i$ when $|i-j| >1$
%\item $T_iT_jT_i = T_jT_iT_j$ when $|i-j| = 1$
%\end{enumerate}

Using the notation from Section \ref{general linear q-group} we then have the quantum Schur--Weyl duality:

\begin{thm} \label{q-Schur-Weyl}
The Hecke algebra $H_r(q)$ acts on the tensor power $V_{q,n}^{\otimes r}$ via the quantum $R$ matrix for $U_{q,n}$. This action commutes with the $U_{q,n}$-module structure on $V_{q,n}^{\otimes r}$ giving homomorphisms $H_r(q) \rightarrow E_{q,n}^r$ which are surjective for all $n$ and isomorphisms for $n \geq r$. 
\end{thm}

This is the main result of \cite{DPS}.

\begin{cor} \label{q-Jones}
Suppose $r \geq \ell$. Then the Hecke algebra $H_r(q)$ has the following semisimple quotients $\overline E_{q,1}^r, \overline E_{q,2}^r, \cdots ,\overline  E_{q, \ell -1}^r$.
\end{cor}
\begin{proof}
By Theorem \ref{q-Schur-Weyl} we have $H_r(q) \simeq E_{q,r}^r$. Then the corollary follows from (\refeq{surj n>m}) and (\refeq{E to barE}) .
\end{proof}

\begin{remarkcounter} \begin{enumerate}
\item 
The semisimple quotients of $H_r(q)$ listed in Corollary \ref{q-Jones} are obvious generalisations of the Jones algebras introduced in \cite{ILZ}, and as explained in the introduction this is the reason why we use the name `higher Jones algebras' in this paper.
 \item In analogy with 
the modular case we see that  $ E_{q,1}^r = k = \overline E_{q,\ell-1}^r $ for all $r$. 
\end{enumerate}
\end{remarkcounter}

The simple modules for $H_r(q)$ are parametrized by the set of $\ell$-regular partitions of $r$. We denote the simple $H_r(q)$-module associated with such a partition $\lambda$ by $D_{q,r}(\lambda)$. Our aim is to derive an algorithm for computing the dimensions of a special class of simple $H_r(q)$-modules, namely those coming from the higher Jones algebras.

In analogy with the notation in Section \ref{class of simple} we set 
$$\overline \Lambda^r(\ell) = \{\lambda = (\lambda_1, \lambda_2, \cdots ,\lambda_m) | \lambda \text { is a partition of $r$ and } \lambda \in A_m(\ell) \text { for some } m < \ell' \}.$$
So $\overline \Lambda^r(\ell) $ consists of those partitions of $r$ which have at most $m < \ell'$ non-zero terms and satisfy $\lambda_1 - \lambda_m \leq \ell' - m$.  Clearly, the partitions in $\overline \Lambda^r(\ell)$ are all $\ell'$-regular.

The result in Proposition \ref{inductive formula} carries over unchanged to the quantum case and leads to the following analogue of Theorem \ref{main symm}.
\begin{thm} \label{main q-symm}
Let $r > 0$ and suppose $\lambda \in \overline \Lambda^r(\ell)$. Then the dimension of the simple $H_r(q)$-module $D_{q,r}(\lambda)$ is recursively determined by 
$$ \dim D_{q,r}(\lambda) = \sum_{i: \lambda - \epsilon_i \in \overline \Lambda^{(r-1)}(\ell)} \dim D_{q,r-1}(\lambda - \epsilon_i).$$
\end{thm}

This theorem allows us to determine the dimensions of a class of simple modules for $H_r(q)$ just like we did for symmetric groups in Section \ref{class of simple}. The only difference is that $\ell$ in contrast to $p$ may now take any value in $\Z_{>0}$. We illustrate by a couple of examples.

\begin{examplecounter}
Let $\ell = 8$, i.e. $q$ is a root of unity of order $8$. In this case $\overline \Lambda^r(8)$ consists of the trivial partition of $r$ (corresponding to the trivial module for $H_r(q)$), the unique $3$-parts partition $\nu$ of $r$ with $\nu_1 - \nu_3 \leq 1$, and the $2$ parts partitions $(s+1,s-1)$ and $(s,s)$, if $r = 2s$ is even, respectively $(s, s-1)$, if $r = 2s-1$ is odd. It is easy to deduce from Theorem \ref{main q-symm} that the partitions with $2$ parts all correspond to simple $H_r(q)$-modules of dimension $2^s$.
\end{examplecounter}

\begin{examplecounter}
Consider the case $\ell = 12$. Here $\overline \Lambda^r(12)$ consists of the trivial partition $(r)$, the unique partition $\nu$ with $5$ parts satisfying $\nu_1 - \nu_5 \leq 1$, the partitions $\lambda^1(r), \lambda^2(r), \lambda^3(r)$  with $2$-parts
$$ \{\lambda^1(r), \lambda^2(r), \lambda^3(r)\} =  \begin{cases} \{(s+2, s-2), (s+1, s-1), (s,s)\} \text { if } r = 2s, \\ \{(s+2, s-1), (s+1,s)\} \text { if } r = 2s +1; \end{cases}$$
the partitions $\mu^1(r), \mu^2(r), \mu^3(r), \mu^r(4)$ with $3$ parts
$$ \{\mu^1(r), \mu^2(r), \mu^3(r), \mu^4(r)\} =  \begin{cases} \{(s+2,s-1,s-1),(s+1, s+1, s-2), (s+1, s, s-1), (s,s,s)\}\\ \text { if } r = 3s, \\ \{(s+2,s, s-1), (s+1,s+1, s-1), (s+1,s,s)\} \text { if } r = 3s +1,\\  \{(s+2,s+1, s-1), (s+2,s, s), (s+1,s+1,s)\} \text { if } r = 3s +2;\\  \end{cases}$$
and the partitions $\eta^1(r), \eta^2(r), \eta^3(r)$ with $4$ parts
$$ \{\eta^1(r), \eta^2(r), \eta^3(r)\} =  \begin{cases} \{(s+1,s+1,s-1,s-1), (s+1,s,s, s-1), (s,s,s,s)\} \text { if } r = 4s, \\ \{(s+1,s+1,s, s-1), (s+1,s,s, s)\} \text { if } r = 4s +1,\\  \{(s+2,s, s,s), (s+1,s+1, s+1,s-1), (s+1,s+1,s,s)\} \text { if } r = 4s +2,\\  \{(s+2,s+1,s,s), (s+1,s+1,s+1,s) \} \text { if } r = 4s + 3.\end{cases}$$
Here the listed partitions involving a zero or a negative number (these occur only for small $r$) should be deleted. In these  cases as well as in the cases where a set with only $2$ elements is listed it is understood that the corresponding or missing $\lambda$, $ \mu$ or $\eta$ does not occur.

We can use Theorem \ref{main q-symm} to compute the dimensions of the simple modules for $ H_r(q)$ where $q$ a root of unity having order $12$.  In Table 5 we have listed the results for the first few values of $r$. As we know that both the trivial partition and the partition $\nu$ always correspond to simple modules of dimension $1$ we have not included these two partitions in the table.
\end{examplecounter} 
\eject 
\centerline{
{ \it Table  5.  Dimensions of simple modules for $H_r(q)$ with $\ell = 12$.}}

\vskip .5cm

\centerline
{\begin{tabular}{ r| c  c c| c c c c | c c c| c}
   & $\lambda^r(1)$ & $\lambda^2(r)$ & $\lambda^3(r)$ & $\mu^1(r)$ & $\mu^2(r)$& $\mu^3(r)$& $\mu^4(r)$ &$\eta^1(r)$ & $\eta^2(r)$ &  $\eta^3(r)$    \\  \hline 
   1 &  &&1 &  &  & 1 &&  &&1 \\ 
  2 &  & 1 &1 & &1&1 &   &1&&1&  \\
  3 && 1 &2 & 1 & & 2 &  1 &2  &1&&\\
  4 &1& 3 & 2 & 3& 2 &3&&2& 3& 1&\\
  5 & 4 & 5& & 5 &6  &5 &&5&4&&  \\
  6 &4 &9 &5& 6 & 5 &16&5 & 4 &5& 9&\\
  7 & 13& 14&&22 & 5 &21 && 13  & 14&\\
 8 & 13 &27&14&27&43 & 26 & & 13 &27 &14&\\
 9 & 40 &41&&43 & 27 & 96 & 26 & 40 & 41&&\\
10 & 40 &81&41&139 & 123 & 122 & & 41 & 40 & 81\\
\end{tabular}}

\section{A class of simples modules for BMW-algebras}

Denote by $x, v, z$  be three indeterminates and set $R = \Z[v, v^{-1}, x, z]/((1-x)z + (v - v^{-1}))$. Let $r \geq 1$ be an integer and consider the general $3$-parameter $BMW$-algebra $BMW_r(R)$ over $R$ as in \cite[Definition 3.1]{Hu}.
As an $R$-module $BMW_r(R)$ is free of rank (2r-1)!! (with basis indexed by Brauer diagrams).

As in the previous sections we denote by $k$ an arbitrary field containing a root of unity $q$ of order $\ell$. We make $k$ into an $R$-algebra by specializing $v$ to $-q^{2n+1}$, $z$ to $q-q^{-1}$, and $x$ to $1 - \sum_{i=-n}^n q^{2i}$. Then the $BMW$-algebra over $k$ that we shall work with is
$$ BMW_r(-q^{2n+1}, q) = BMW_r(R) \otimes_R k.$$
For $q = 1$ it turns out that $BMW_r(-q^{2n+1}, q)$ may be identified with the Brauer algebra $\mathcal B_r(-2m)$, see the remarks after Definition 3.1 in \cite{Hu}. We treated the Brauer algebras in Section \ref{Brauer} so in this section we shall assume $q \neq 1$.

Using the notation from Section \ref{q-sp} the quantum analogue  \cite[Theorem 1.5]{Hu} of the Schur--Weyl duality for symplectic groups says:

\begin{thm} \label{qsp-Schur-Weyl}
The algebra $BMW_r(-q^{2n+1}, q)$  acts naturally on the tensor power $V_{q,n}^{\otimes r}$. This action commutes with the $U_{q,n}$-module structure on $V_{q,n}^{\otimes r}$ giving homomorphisms $BMW_r(-q^{2n+1},q) \rightarrow E_{q,n}^r$ which are surjective for all $n$. 
\end{thm}

\begin{cor} \begin{enumerate}
\item  If $\ell$ is odd, 
 then the $BMW$-algebra $BMW_r(-q^{2n+1},q)$ surjects onto the semisimple algebra  $\overline E_{q,n}^r$ for $n= 1, 2, \cdots , (\ell - 1)/2$ and $r> 0$.
\item  If $\ell$ is even, 
 then the $BMW$-algebra $BMW_r(-q^{2n+1},q)$ surjects onto the semisimple algebra  $\overline E_{q,n}^r$ for $n= 1, 2, \cdots , (\ell - 4)/2$ and $r> 0$. 
\end{enumerate}
\end{cor}
\begin{proof}
By Theorem \ref{qsp-Schur-Weyl} we have $BMW_r(-q^{2n+1}, q)$ surjects onto $ E_{q,n}^r$ for all $n, r$ and hence also on $\overline E_{q,n}^r$ by (\refeq {surj E to Ebar qsp}). In Section \ref{q-sp} we observed that these latter algebras are non-zero for the $n$'s listed in the corollary. 
\end{proof}

Let  $\lambda$ is a partition of $r-2i$ for some $i \leq r/2$. In analogy with the Brauer algebra case we denote the simple $BMW_r(-q^{2n+1}, q)$-module corresponding to $\lambda$ by $D_{BMW_r(-q^{2n+1}, q)}(\lambda)$.
Recall the definition of $A_n(\ell)$ from Section \ref{q-sp} and set for any $r>0$
$$ \overline \Lambda^r(n, \ell) =  (\Lambda^r \cup \Lambda^{r-2} \cup \cdots ) \cap A_{n}(\ell).$$
Then arguments similar to the ones used above give

\begin{thm} \label{main BMW}
Let $r > 0$.
\begin{enumerate}
\item Suppose $\ell$ is odd. Let $n \in \{1, 2, \cdots (\ell - 1)/2\}$ and $\lambda \in \overline \Lambda^r(n, \ell )$. Then the dimension of the simple $BMW_r(-q^{2n+1}, q)$-module $D_{BMW_r(-q^{2n+1}, q)}(\lambda)$ is recursively determined by 
$$ \dim D_{BMW_r(-q^{2n+1}, q)}(\lambda) = \sum_{i: \lambda \pm \epsilon_i \in \overline \Lambda^{(r-1)}(\delta, \ell)} \dim D_{BMW_{r-1}(-q^{2n+1}, q)}(\lambda \pm \epsilon_i).$$

\item Suppose $\ell$ is even. Let $n \in \{1, 2, \cdots (\ell - 4)/2\}$ and $\lambda \in \overline \Lambda^r(n, \ell )$. Then the dimension of the simple $BMW_r(-q^{2n+1}, q)$-module $D_{BMW_r(-q^{2n+1}, q)}(\lambda)$ is recursively determined by 
$$ \dim D_{BMW_r(-q^{2n+1}, q)}(\lambda) = \sum_{i: \lambda \pm \epsilon_i \in \overline \Lambda^{(r-1)}(\delta, \ell)} \dim D_{BMW_{r-1}(-q^{2n+1}, q)}(\lambda \pm \epsilon_i).$$
\end{enumerate}
\end{thm}

\begin{examplecounter}
We shall illustrate Theorem \ref{main BMW} in the case $\ell$ is even (the odd case is equivalent to the Brauer case in Section \ref{Brauer}). So we take $\ell = 10$. Then the relevant values of $n$ are $1, 2$ and $3$. The weight set $\overline \Lambda^r(1,10)$ contains $2$ elements  (except for $r =1$)  $\lambda^1(r), \lambda^2(r)$, namely $ (2), (0)$, when $ r$ is even, and $ (3), (1)$, when $r$ is odd. Likewise, $\overline \Lambda^r(2,10)$ contains $2$ elements  when $r$ is odd (except for $r=1$) and $4$ elements, when $r$ is even (except for $r = 2$). We denote these weights $\mu^1(r), \mu^2(r), (\mu^3(r), (\mu^4(r))$. They are $(2,2), (2,0), (1,1), (0,0)$, when $ r$ is even, and $ (2,1), (1,0)$, when $r$ is odd. Finally, $\overline \Lambda^r(3,10)$ consists of $2$ elements $\nu^1(r), \nu^2(r)$, namely $(1,1,0), (0,0,0)$, when $r$ is even, and $(1,1,1), (1,0,0)$, when $r$ is odd (except $r=1$).

In Table 6 we have in row $r$ listed (in the order given above) the dimensions of the simple modules for $ B_r(-q^{2n+1}, q) $ for $r \leq 10$. These numbers are computed recursively using Theorem \ref{main BMW}.
\end{examplecounter}

{ \it Table  6.  Dimensions of simple modules for $BMW_r(-q^{2n+1}, q)$ when $\ell = 10$ and $n = 1 , 3, 5$.}
\vskip .5cm

\centerline {
\begin{tabular}{ r| c  c  |c c c c |c c| c } 
         &$n=1$&&&$n=2$&&&$n=3$& \\ \hline
        r & $\lambda^1(r)$ & $\lambda^2(r)$ & $\mu^1(r)$ & $\mu^2(r)$ & $\mu^3(r)$& $\mu^4(r)$ & $\nu^1(r)$ & $ \nu^2(r)$   \\   \hline 
  1 &  & 1&  & 1 & &&&1 & \\ 
  2 & 1 & 1 &  1& 1 & 1 & 1 & 1 & 1  \\ 
  3 & 1&2& 2& 3& &  & 1 & 2  \\
  4 & 3 &2 & 2& 5 & 5 &3& 3& 2 \\
  5 &3 & 5 & 12 & 13 & & &3 & 5\\
  6 & 8 & 5& 12& 25&25 &13&8 & 5  \\
  7 &8 &13 &62 &63&&  & 8 & 13\\
  8 & 21& 13&62 &125 & 125 &63& 21 & 13\\
 9 & 21 &44&312&313 &&  & 21 & 34\\
10 & 65 &44&312&625 & 625 & 313&55 & 34\\
\end{tabular}}

\vskip 1 cm


\begin{thebibliography}{}

\bibitem{A80a} H.H.~Andersen, The strong linkage principle, J. Reine Ang. Math. 315 (1980), 53-59.

\bibitem{A92} H. H. Andersen, Tensor products of quantized tilting modules, Comm. Math. Phys. 149 (1992), 149 – 159.
\bibitem{A03} H.H.~Andersen, The strong linkage principle for quantum groups at roots of 1, J. Alg. 260 (2003), 2-15.
\bibitem{A17} H. H.~Andersen, Simple modules for Temperley--Lieb algebras and related algebras,  J. Alg. 520 (2019), 276-308.

\bibitem{AP} H. H. Andersen and J. Paradowski, Fusion categories arising from semisimple Lie algebras,Comm. Math. Phys. 169 (1995), 563 – 588.
\bibitem{APW} H.H.~Andersen, P. Polo and Wen Kexin, Representations of quantum algebras, Invent. Math. 104 (1991), 1 - 59.
\bibitem{AS} H. H. Andersen and C.~Stroppel,  Fusion rings for quantum groups, Algebras and Representation Theory 17 (2014), 1869 - 1888.
\bibitem{AST1} H. H. ~Andersen, C. Stroppel and D.~Tubbenhauer, Cellular structures using Uq-tilting modules, Pacific Journal of Mathematics 292 (2018), 21-59.
\bibitem{AST2} H. H. ~Andersen, C. Stroppel and D.~Tubbenhauer, Semisimplicity of Hecke and (walled) Brauer algebras,  J. Aust. Math. Soc. 103 (2017), 1-44.


\bibitem{B} R. Brauer, On algebras which are connected with the semisimple continuous groups. Ann. of Math. 38 (1937) 857-872, 1937.
\bibitem{CL} R. Carter and G. Lusztig, On the modular representations of the general linear and symmetric groups, Math. Z. 136 (1974), 193-242.
\bibitem{DDH} R. Dipper, S. Doty and Jun Hu, Brauer algebras, symplectic Schur algebras and Schur--Weyl duality, Trans. AMS 360 (2008), 189-213.
\bibitem{Do-book} S. Donkin, Rational Representations of Algebraic Groups, Lecture Notes in Math. 1140 (Springer 1985).
\bibitem{Do} S. Donkin, On tilting modules for algebraic groups, Math. Z. 212 (1993), 39-60.
\bibitem{DH} S. Doty and J. Hu, Schur--Weyl duality for orthogonal groups, Proc. LMS 98 (2008), 679-713.
\bibitem{DPS} J. Du, B. Parshall and L. Scott, Quantum Weyl reciprocity and tilting modules, Comm. Math. Phys. 195 (1998), 321-352.
%\bibitem{ES} M. Ehrig and C. Stroppel, Schur-Weyl duality for the Brauer algebra and the ortho-symplectic Lie superalgebra, Math. Z. 284 (2016),  595–613.
\bibitem{GW}  F. M. Goodman and H. Wenzl, Littlewood-Richardson coefficients for Hecke algebras at roots of unity, Adv. Math. 82 (1990), 244-265.
\bibitem{GL} J. Graham and G. Lehrer, Cellular algebras, Invent. Math. 123 (1996), 1-34.

\bibitem{Hu} J. Hu, BMW algebra, quantized coordinate algebra and type C Schur--Weyl duality,  Repr. Theory 15 (2011), 1-62.

\bibitem{ILZ}K. Iohara, G.I. Lehrer, R.B. Zhang, Temperley--Lieb at roots of unity, a fusion category and the Jones quotient, Math Res. Lett. (to appear), online available on arXiv: 1707.01196.
\bibitem{RAG} J.C.~Jantzen, {\em Representations of Algebraic Groups}, Mathematical Surveys and Monographs 107, Second edition, American Mathematical Society (2003).
\bibitem{Kl} A. Kleshchev, Completely splittable representations of symmetric groups, J. Algebra 181 (1996), 584-592.
%\bibitem{Lu-book} G. Lusztig, Introduction to quantum groups. Modern Birkh¨auser Classics. Birkh¨auser/Springer, New York,
%2010. Reprint of the 1994 edition.
\bibitem{Ma} O. Mathieu, Filtrations of G-modules, Ann. scient. Éc. Norm. Sup. 23 (1990), 625-644.
\bibitem{Ma1} O. Mathieu, On the dimension of some modular irreducible representations of the symmetric group, Lett. Math. Phys. 38 (1996), 23-32.
\bibitem{Wa} J.-p.~Wang, Sheaf cohomology on G/B and tensor products of Weyl modules, J. Algebra 77 (1982), 162-185.
\bibitem{We1} H. Wenzl, Hecke algebras of type $A_n$ and subfactors, Invent. math. 92, 349-383 (1988). 
\bibitem{We2}H. Wenzl, On the Structure of Brauer's Centralizer Algebras, Annals of Mathematics 128 (1988), 173-193.
\bibitem{We3} H. Wenzl, Quantum Groups and Subfactors of Type $B$, $C$, and $D$, Commun. Math. Phys. 133, 383-432 (1990).
 
\end{thebibliography}
\end{document}